\documentclass[10pt]{amsart}

\usepackage{latexsym,mathrsfs}
\usepackage{color}

\usepackage[english]{babel}
\usepackage{paralist}
\usepackage{stmaryrd}
\usepackage{amsfonts,amssymb,amsmath}
\usepackage{amsthm}
\usepackage[latin1]{inputenc}
\usepackage{url}
\usepackage{hyperref}

\newcommand{\eps}{\varepsilon}
\newcommand{\R}{\mathbb{R}}

\newcommand{\C}{\mathbb{C}}
\newcommand{\N}{\mathbb{N}}

\newcommand{\Z}{\mathbb{Z}}

\newcommand{\es}[1]{\begin{equation}\begin{split}#1\end{split}\end{equation}}
\newcommand{\est}[1]{\begin{equation*}\begin{split}#1\end{split}\end{equation*}}
\newcommand{\as}[1]{\begin{align}#1\end{align}}

\newcommand{\II}{\mathcal{J}}

\newcommand{\seq}{\subseteq}
\newcommand{\cdotsp}{\cdots\hspace{-0.03em}}

\newtheorem*{theo*}{Theorem}
\newtheorem{theo}{Theorem}

\newtheorem{ezer}{Exercise}

\newtheorem{prop}[ezer]{Proposition}

\newtheorem{lemma}{Lemma}
\newtheorem{conj}{Conjecture}[section]

\newtheorem*{remark}{Remark}

\newtheorem*{rem*}{Remark}
\usepackage[colorinlistoftodos,french,bordercolor=black,backgroundcolor=red,linecolor=red,textsize=footnotesize,textwidth=0.75in,shadow]{todonotes}

\def\sumprime{\operatornamewithlimits{\sum\nolimits^\prime}}

\newcommand{\pr}[1]{\left( #1\right)}

\newcommand{\Res}{\operatorname*{Res}}

\newcommand{\ba}{{\boldsymbol a}}
\newcommand{\be}{{\boldsymbol e}}
\newcommand{\bd}{{\boldsymbol d}}
\newcommand{\bepsilon}{{\boldsymbol{\epsilon}}}
\newcommand{\bone}{{\boldsymbol 1}}
\newcommand{\bl}{{\boldsymbol \ell}}
\newcommand{\balpha}{\boldsymbol\alpha}
\newcommand{\bbeta}{\boldsymbol\beta}

\newcommand{\bx}{\boldsymbol{x}}

\newcommand{\bv}{\boldsymbol{v}}

\newcommand{\bz}{\boldsymbol{z}}
\newcommand{\bc}{\boldsymbol{c}}
\newcommand{\bff}{\boldsymbol{f}}
\newcommand{\bb}{\boldsymbol{b}}
\newcommand{\bg}{\boldsymbol{g}}
\newcommand{\bzero}{\boldsymbol{0}}

\newcommand{\br}{\boldsymbol{r}}

\newcommand{\comment}[1]{}

\newcommand{\Smt}{\mathfrak {S}_1}

\let\originalleft\left
\let\originalright\right
\renewcommand{\left}{\mathopen{}\mathclose\bgroup\originalleft}
\renewcommand{\right}{\aftergroup\egroup\originalright}

\numberwithin{equation}{section}

\author{Sandro Bettin}
\address{Sandro Bettin, DIMA - Dipartimento di Matematica, Via Dodecaneso, 35, 16146 Genova - ITALY}
\email{bettin@dima.unige.it}
\author{Kevin Destagnol}
\address{Kevin Destagnol, Max Planck Institut f\"ur Mathematik, Vivatsgasse 7, 53111 Bonn - GERMANY}
\email{kdestagnol@mpim-bonn.mpg.de}

\begin{document}

\title[The Manin-Peyre's conjectures for $x_1y_2y_3+x_2y_1y_3+x_3y_1y_2=0$]{The power-saving Manin-Peyre's conjectures \\ for a senary cubic}

\begin{abstract}
Using recent work of the first author~\cite{Sandro1}, we prove a strong version of the Manin-Peyre's conjectures with a full asymptotic and a power-saving error term for the two varieties respectively in $\mathbb{P}^2 \times \mathbb{P}^2$ with bihomogeneous coordinates $[x_1:x_2:x_3],[y_1:y_2,y_3]$ and in $\mathbb{P}^1\times \mathbb{P}^1 \times \mathbb{P}^1$ with multihomogeneous coordinates $[x_1:y_1],[x_2:y_2],[x_3:y_3]$ defined by the same equation $x_1y_2y_3+x_2y_1y_3+x_3y_1y_2=0$. We thus improve on recent work of Blomer, Br\"udern and Salberger \cite{BBS} and provide a different proof based on a descent on the universal torsor of the conjectures in the case of a del Pezzo surface of degree 6 with singularity type $\mathbf{A}_1$ and three lines (the other existing proof relying on harmonic analysis \cite{CLT}). Together with~\cite{Blomer2014} or with recent work of the second author \cite{Dest2}, this settles the study of the Manin-Peyre's conjectures for this equation.
\end{abstract}

\maketitle

\tableofcontents

\section{Introduction}

In the late 80s, Manin and his collaborators \cite{FMT} proposed a precise conjecture predicting, for smooth Fano varieties, the behaviour of the number of rational points of bounded height (with respect to an anticanonical height function) in terms of geometric invariants of the variety. The conjecture was later generalised by Peyre \cite{P95} to ``\textit{almost Fano}" varieties in the sense of~\cite[D\'efinition 3.1]{P95}.
\begin{conj}[\textbf{Manin, 1989}]
Let $V$ be an ``almost Fano" variety in the sense of \cite[\textit{D\'efinition 3.1}]{P95} with $V(\mathbb{Q}) \neq \varnothing$ and let $H$ be an anticanonical height function on $V(\mathbb{Q})$. Then there exists a Zariski open subset $U$ of $V$ and a constant $c_{H,V}$ such that, for $B \geqslant 1$,
$$
N_{U,H}(B):=\#\{x \in U(\mathbb{Q}) \mid H(x) \leqslant B\}=c_{H,V}B\log(B)^{\rho-1}(1+o(1)),
$$
where $\rho={\rm{rank}}({\rm{Pic}}(V))$.
\label{manin1}
\end{conj}

Peyre~\cite{P95}, and then Batyrev and Tschinkel \cite{BT2} and Salberger \cite{Sal} in a more general setting, also proposed a conjectural expression for the constant $c_{H,V}$ in terms of geometric invariants of the variety. We do not record this conjecture here in any more details and refer the interested reader to \cite{P03} for example. There are a number of refinements of the Manin-Peyre's conjectures and we will focus throughout this paper on the following one \cite{Tim}.

\begin{conj}[\textbf{Refinement of the Manin-Peyre's conjectures}]
Let $V$ be an ``almost Fano" variety in the sense of \cite[\textit{D\'efinition 3.1}]{P95} with $V(\mathbb{Q}) \neq \varnothing$ and let $H$ be an anticanonical height function on $V(\mathbb{Q})$. Then there exists a Zariski open subset $U$ of $V$, a polynomial $P_{U,H}$ of degree $\rho$ and $\delta\in]0,1[$ such that, for $B \geqslant 1$
$$
N_{U,H}(B)=c_{H,V}BP_{U,H}(\log(B))+O\left(B^{1-\delta}\right),
$$
where $\rho={\rm{rank}}({\rm{Pic}}(V))$ and the leading coefficient of $P_{U,H}$ agrees with Peyre's prediction.
\label{manin2}
\end{conj}

There has been very little investigations on the lower order coefficients and this seems to be a difficult question but the examples we study in this paper might be an interesting testing ground.\\
\indent
These two conjectures have been the center of numerous investigations in the past few years using techniques from harmonic analysis in the case of equivariant compactifications of some algebraic groups (see for example~\cite{BT,TT}) or from analytic number theory and the circle method in the case where the number of variables is large enough with respect to the degree (see for example \cite{Bi,BHB}). In the remaining cases, the only available method relies on a combination of analytic number theory or geometry of number and on a descent on some quasi-versal torsors in the sense of \cite{CS}. 
Most of these investigations (especially in cases relying on a descent) are concerned with surfaces (see for example works of Browning, La Bret\`eche, Derenthal and Peyre~\cite{20,35,BrD}), whereas very little is known in higher dimensions.
In particular there are only very few examples of Fano varieties in higher dimensions for which Conjecture~\ref{manin2}, or even Conjecture~\ref{manin1}, are known to hold using such a descent method (see~\cite{19,Schmidt,Blomer2014,Dest2}). The goal of this paper is to give another such example.\\
\newline
\indent
In this paper we shall consider the solutions to the equation
\begin{equation}
x_1y_2y_3\cdots y_n+x_2y_1y_3\cdots y_n+\cdots+x_ny_1y_2\cdots y_{n-1}=0.
\label{eq}
\end{equation}
Notice that, upon excluding the points for which $y_1\cdots y_n=0$, one can also rewrite the above equation as a linear equations between fractions
\est{
\frac{x_1}{y_1}+\cdots+ \frac{x_n}{y_n}=0.
}
We shall focus on the case $n=3$ in the present paper. The cases $n \geqslant 4$ will be the subject of future work.

One can view equation~\eqref{eq} in three natural ways.
First, one can consider the singular projective hypersurface of $\mathbb{P}^{2n-1}$ with homogeneous coordinates $[x_1:\cdots:x_n:y_1:\cdots:y_n]$ defined by~\eqref{eq}. This was done in 2014 by Blomer, Br\"udern and Salberger who in~\cite{Blomer2014} proved Conjecture~\ref{manin2} for $n=3$ using a combination of lattice point counting and analytic counting by multiple Mellin integrals. This setting was also studied by the second author~\cite{Dest2} who, by elementary counting methods, proved Conjecture~\ref{manin1} for the case $n\geq2$ for the following anticanonical height function
$$
H\left([x_1:\cdots:x_n:y_1:\cdots:y_n]\right)=\max_{1 \leqslant i \leqslant n}\max\{|x_i|,|y_i|\}^n.
$$
It is likely that a generalization of the methods of \cite{Blomer2014} or of the present paper allows one to obtain Conjecture \ref{manin2} for every $n \geqslant 4$ as well. It is worth noticing that in this case, the varieties under consideration are equivariant compactifications of the algebraic groups $\mathbb{G}_a^{n-1} \times \mathbb{G}_m^{n-1}$. Harmonic analysis techniques might also be able to handle this case. To our knowledge this hasn't been done so far, but it would be interesting to compare the two approaches.

One can also think of~\eqref{eq} as defining the singular biprojective variety $\widetilde{W}_n$ of $\left(\mathbb{P}^{n-1}\right)^2$ with bihomogeneous coordinates $[x_1:\cdots:x_n],[y_1:\cdots:y_n]$ defined by the equation~\eqref{eq}. An anticanonical height function is then given by
$$
\widetilde{H}\left([x_1:\cdots:x_n],[y_1:\cdots:y_n]\right)=\max_{1\leqslant i \leqslant n}|x_i|^{n-1}\max_{1\leqslant i \leqslant n}|y_i|.
$$
In this case, the varieties under consideration are not equivariant compactifications, the rank of the Picard group of $\widetilde{W}_n$ is $2^n-n$ and the subset where $x_1\cdots x_ny_1 \cdots y_n=0$ is an accumulating subset. In recent work \cite{BBS}, Blomer, Br\"udern and Salberger showed that Conjecture \ref{manin1} holds for $\widetilde{W}_3$ using Fourier analysis. Using recent results of the first author \cite{Sandro2}, we are able to refine the aforementioned result~\cite{BBS} proving the stronger Conjecture \ref{manin2} for $\widetilde{W}_3$.

\begin{theo}\label{mtt1}
Let $\widetilde{U}$ be the Zariski open subset of $\widetilde{W}_3$ given by the condition $x_1x_2x_3y_1y_2y_3 \neq 0$. There exist $\xi_1>0$ and a polynomial $P_1$ of degree 4 such that
$$
N_{\widetilde{W}_3,\widetilde{H}}(B):=\#\left\{\left([x_1:\cdots:x_n],[y_1:\cdots:y_n]\right) \in \widetilde{U}(\mathbb{Q}) \mid \widetilde{H}(\mathbf{x},\mathbf{y}) \leqslant B\right\}=BP_1(\log(B))+O\left(B^{1-\xi_1}\right).
$$
The leading coefficient of $P_1$ is equal to $\frac{\mathfrak S_1 \cdot \mathcal I}{144}$, where 
\est{
\mathfrak S_1&:=\prod_{p}\pr{1-\frac1p}^5\pr{1+\frac 5p+\frac 5{p^2}+\frac1{p^3}},\\
\mathcal I&:=\iint_{{[-1,1]^3\times[0,1]^2}}\chi_{[0,1/|z|]}\Big(\frac{x_1}{y_1}+\frac{x_2}{y_2}\Big)\,dx_1dx_2dz\cdot\frac{dy_1dy_2}{y_1y_2}=\pi^2 + 24 \log 2-3
}
and $\chi_X$ denotes the characteristic function of a set $X$.
\end{theo}
The work of~\cite{BBS} shows that $\frac{\mathfrak S_1 \cdot \mathcal I}{144}$ coincides with Peyre's prediction for this variety and so Theorem~\ref{mtt1} gives Conjecture \ref{manin2} for $\widetilde{W}_3$.

Finally, a third interpretation of~\eqref{eq} is as the singular subvariety $\widehat{W}_n$ of $\left(\mathbb{P}^1\right)^n$ with multihomogeneous coordinates $[x_1:y_1],\dots,[x_n:y_n]$. The only record of study of an analogous equation is from La Bret\`eche \cite{Br} but with a non anticanonical height function. An anticanonical height is given in this setting by
$$
\widehat{H}\left(  [x_1:y_1],\dots,[x_n:y_n]\right)=\prod_{i=1}^n \max\{|x_i|,|y_i|\}.
$$
We prove the two following theorems, which, combined, gives that Conjecture \ref{manin2} holds for $\widehat{W}_3$.
\begin{theo}
Let $\widehat{U}$ be the Zariski open of $\widehat{W}_3$ defined by the condition $y_1y_2y_3 \neq 0$. Then there exist $\xi_2>0$ and a polynomial $P_2$ of degree 3 such that
$$
N_{\widehat{W}_3,\widehat{H}}(B):=\#\left\{\left([x_1:y_1],[x_2:y_2],[x_3:y_3]\right) \in \widehat{U}(\mathbb{Q}) \mid \widehat{H}(\mathbf{x},\mathbf{y}) \leqslant B\right\}=BP_2(\log(B))+O\left(B^{1-\xi_2}\right).
$$
The leading coefficient of $P_2$ 
is equal to $\frac{\mathfrak S_2 \cdot \mathcal I}{144}$, where $\mathcal I$ is as in Theorem~\ref{mtt1} and
\est{
\mathfrak S_2&:=\prod_{p}\pr{1-\frac1p}^4 \pr{1+ \frac4p + \frac1{p^2}}.\\
}
\label{2}
\end{theo}

\begin{theo}
The variety $\widehat{W}_3$ is isomorphic to a del Pezzo surface of degree 6 with singularity type~$\mathbf{A}_1$ and 3 lines over $\mathbb{Q}$ and the leading constant of the polynomial $P_2$ in Theorem \ref{2} agrees with Peyre's prediction.
\label{3}
\end{theo}

We remark that by Theorem~\ref{3} and~\cite{Dan} one has that $\widehat{W}_3$ is an equivariant compactification of $\mathbb{G}_a^2$. In particular Theorem \ref{2} follows from the more general work of Chambert-Loir and Tschinkel~\cite{CLT}. 

The purpose of giving a new independent proof of Theorem~\ref{2} is double. First, the method presented here uses a descent on the versal torsor and thus it is different from the method in~\cite{CLT} which relies on harmonic analysis techniques and the study of the height zeta function. To our knowledge this is the first time that a full asymptotic with a power-saving error term is obtained on this del Pezzo surface by means of a descent on the versal torsor. The best result using such a method can be found in \cite[Chapter 5]{Tim} were Browning obtains a statement somewhere in between Conjectures~\ref{manin1} and~\ref{manin2}.\\
\indent
Secondly, following the same approach for proving Theorem~\ref{mtt1} and~\ref{2} allows one to appreciate the difference in the structure of the main terms in these two cases, showing how the extraction of the main term in the first case becomes substantially harder as well as allowing the use the proof of Theorem~\ref{2} as a guide for that of Theorem~\ref{mtt1}.

We remark that a generalization of the methods of this paper and of~\cite{BB14} should be within reach and would allow us to obtain a full asymptotic with a power-saving error term for each of these three embeddings and for every $n\geqslant 4$. This will be the focus of future work. We conclude these comments by the following remark about Theorem \ref{2}.

\begin{remark}
We prove Theorem~\ref{2} for any $\xi_2<0.00228169\dots$ One can easily give an explicit power saving also in the case of Theorem~\ref{mtt1} as well as improving the allowed range for $\xi_2$, but in order to simplify the presentation we choose not to do so, since in any case the values obtained could be greatly improved by tailoring the methods of~\cite{Sandro2} to these specific problems.
\end{remark}

The proofs of Theorem~\ref{mtt1} and~\ref{2} roughly proceed as follows. We use the same unique factorization as in recent work of the second author \cite{Dest2} to parametrize the counting problem combined with recent work of the first author~\cite{Sandro2}. More precisely, by means of a descent on the versal torsor we can transform the problem of counting solutions to~\eqref{eq} to that of counting solutions to $a_1x_1z_1+a_2x_2z_3+a_3x_3z_3=0$ with some coprimality conditions, with certain restraints on the sizes of $x_i,z_j$ (depending on the height we had originally chosen), and with $a_1,a_2,a_3$ that can be thought of being very small. By~\cite{Sandro2} (see also~\cite{Sandro1}) we have the meromorphic continuation for the ``parabolic Eisenstein series''
\est{
\sum_{n_1,n_2,n_3,m_1,m_2,m_3 \in \mathbb{Z}_{> 0} \atop a_1m_1+a_2m_2+a_3m_3=0}\hspace{-1em}\frac{\tau_{\alpha_1,\beta_1}(m_1)\tau_{\alpha_2,\beta_2}(m_3)\tau_{\alpha_3,\beta_3}(m_3)}{(m_1m_2m_3)^s},\quad \Re(s)>\tfrac23-\min(\Re(\alpha_i),\Re(\beta_i))\ \ \forall i=1,2,3
}
where $\tau_{\alpha_i,\beta_i}(m)=\sum_{d_1d_2=m}d_1^{-\alpha_i} d_2^{-\beta_i}$ for $(\alpha_1,\alpha_2,\alpha_3), (\beta_1,\beta_2,\beta_3) \in \mathbb{C}^3$. Using this we obtain that in both cases is given by a certain multiple complex integral of the products of $\Gamma$ and $\zeta$ functions, up to a power saving error term. The main part of the paper is then devoted to the use of complex analytic methods to extract the main terms from such integrals. 
{This process is reminiscent of the work~\cite{Br2001} of La Bret\`eche, where he showed how to deduce asymptotic formulas for generic  arithmetic averages from the analytic properties of their associated Dirichlet series. However, his work is not directly applicable to our case. Indeed, in his setting all variables are summed in boxes, whereas in our case the main action happens at complicated hyperbolic spikes. Of course, one could use La Bret\`eche's work in combination with some suitable version of the hyperbola method, but in fact this would not simplify substantially the problem and would still eventually require arithmetic and complex-analytic computations essentially equivalent to ours. For this reason, we prefered to approach the relevant sums in a more direct way.
}

The paper is structured as follows. First, in Section~\ref{par} we reparametrize the solutions to~\eqref{eq} using a descent on the versal torsor. In Section~\ref{geom} we prove Theorem~\ref{3}. In Section~\ref{parab} we state the main Lemma on the parabolic Eisenstein series and a smoothing lemma useful to avoid problems of sharp cut-offs. Then, in Section~\ref{nocop},~\ref{arith} and~\ref{pmt} we prove Theorems~\ref{2} and~\ref{mtt1} in three steps of increasing difficulties: first Theorem~\ref{2} without the aforementioned coprimality conditions, then we include these conditions and finally we prove Theorem~\ref{mtt1}.

\section*{Notations}
We use the vector notation $\bv=(v_1,\dots,v_k)$ where the dimension is clear from the context. Also, given a vector $\bv\in\C^k$ and $c\in\C$ with $\bv+c$ we mean $(v_1+c,\dots,v_k+c)$. With $\iint$ we indicate the integration with respect to several variables, whose number is clear from the context.
For $c\in\R$, with $\int_{(c)}$ we indicate that the integral is taken along the vertical line from $c-i\infty$ to $c+i\infty$. Also, we indicate with $c_{z}$ the line of integration corresponding to the variable $z$. 
Given, $a_1,\dots,a_k\in\Z$ we indicate the GCD and the HCF of $a_1,\dots,a_k$ by $(a_1,\dots,a_k)$ and $[a_1,\dots, a_k]$ respectively.

We indicate the real and imaginary part of a complex number $s\in\C$ by $\sigma$ and $t$ respectively, so that $s=\sigma+i t$. Also, $\eps$ will denote an arbitrary small and positive real number, which is assumed to be sufficiently small and upon which all bounds are allowed to depend.
{Finally, in Section~\ref{pmt}, we denote by $C_1,C_2,C_3,\dots$ a sequence of fixed positive real numbers.}
\section{The descent on the versal torsor}\label{par}
For $n \geqslant 2$, we let $N=2^n-1$. For every $h \in \{ 1,\dots,N\}$, we denote its binary expansion by
$$
h=\sum_{1 \leqslant j \leqslant n} \epsilon_j(h)2^{j-1},
$$ 
with $\epsilon_j(h) \in \{0,1\}$. We will let $s(h)=\sum_{j \geqslant 1} \epsilon_j(h)$ be the sum of the bits of $h$. We will say that a integer $h$ is dominated by $\ell$ if for every $j \in \mathbb{N}$, we have $\epsilon_j(h) \leqslant \epsilon_j(\ell)$. We will use the notation $h \preceq \ell$ to indicate that $h$ is dominated by $\ell$. We will say that an $N-$tuple $(z_{1},\dots,z_N)$ is reduced if $\gcd(z_h,z_{\ell})=1$ when $h \not\preceq \ell$ and $\ell \not\preceq h$.\\
\indent
We give the following lemma which gives a unique factorization for the variables $y_i$ inspired by \cite{5,4} and \cite{Br} and which will be very useful to parametrize rational solutions of~\eqref{eq}.

\begin{lemma}[\textbf{\cite[\textbf{Lemme 1}]{Dest2}}]
There is a one-to-one correspondence between the $n-$tuples of non negative integers $(y_i)_{1 \leqslant i \leqslant n}$ and the reduced $N-$tuples $(z_h)_{1 \leqslant h \leqslant N}$ of non negative integers such that
$$
\forall j \in \llbracket 1,n \rrbracket, \quad y_j=\prod_{1 \leqslant h \leqslant N} z_h^{\epsilon_j(h)} \quad \mbox{and} \quad [y_1,\dots,y_n]=\prod_{1 \leqslant h \leqslant N} z_h.
$$
\label{lemmefacto}
\end{lemma}
\subsection{The case of ${\widehat{W}}_n$}
\label{ut1}
We want to estimate, for $B \geqslant 1$, the quantity
$$
N_{\widehat{W}_n,\widehat{H}}(B):=\#\left\{\left([x_1:y_1],[x_2:y_2],[x_3:y_3]\right) \in \widehat{U}(\mathbb{Q}) \mid \widehat{H}(\mathbf{x},\mathbf{y}) \leqslant B\right\}.
$$
Clearly, we have
$$
N_{\widehat{W}_n,\widehat{H}}(B)=\#\left\{(\mathbf{x},\mathbf{y})\in \mathbb{Z}^n \times \mathbb{Z}_{>0}^n  \hspace{1mm}:\hspace{1mm} 
\begin{array}{l}
\hspace{5em}\widehat{H}(\mathbf{x},\mathbf{y}) \leqslant B\\
(\mathbf{x},\mathbf{y}) \mbox{ satisfies }~\eqref{eq},\
\gcd(x_i,y_i)=1
\end{array}
\right\}.
$$
Using Lemma \ref{lemmefacto} the equation~\eqref{eq} can be rewritten as
\begin{equation}
\sum_{j=1}^n d_jx_j=0
\quad
\mbox{with}
\quad 
 d_i=\prod_{1 \leqslant h \leqslant N} z_h^{1-\epsilon_i(h)}\quad \forall i \in \llbracket 1,n \rrbracket.
\label{torsor}
\end{equation}
We then obtain the divisibility relation $z_{2^{j-1}} \mid x_j$ for every $j \in \{ 1,\dots,n \}$. Since we have the conditions $\gcd(x_j,y_j)=1$ and $z_{2^{j-1}} \mid y_j$, we can deduce that for every $j \in \{ 1,\dots,n \}$, $z_{2^{j-1}}=1$. Finally, we have
$$
N_{\widehat{W}_n,\widehat{H}}(B)=\#\left\{(\mathbf{x},\mathbf{z})\in \mathbb{Z}^n \times \mathbb{Z}_{>0}^{N-n}  \hspace{1mm}:\hspace{1mm} 
\begin{array}{l}
\displaystyle\prod_{i=1}^n \max\bigg\{|x_i|,\prod_{1 \leqslant h \leqslant N} |z_h|^{\epsilon_i(h)}\bigg\} \leqslant B\\[3mm]
(z_h)_{1 \leqslant h \leqslant N} \mbox{ reduced}, \quad
\displaystyle\sum_{i=1}^n x_id_i=0
\end{array}
\right\}.
$$
In the case $n=3$, renaming for simplicity $z_6$ by $z_1$, $z_5$ by $z_2$ and $z_7$ by $z_4$, one gets the following expression for $N_{\widehat{W}_3,\widehat{H}}(B)$:
\es{\label{expwh}
&N_{\widehat{W}_3,\widehat{H}}(B)=\\
&\hspace{2em}\#\left\{\left(\mathbf{x},\mathbf{z}\right) \in \mathbb{Z}^{3}\times \mathbb{Z}_{>0}^4 : \begin{array}{l}
 \gcd(x_1,z_2z_3z_4)= \gcd(x_2,z_1z_3z_4)= \gcd(x_3,z_1z_2z_4)=1\\
 \gcd(z_1,z_2)=\gcd(z_1,z_3)=\gcd(z_2,z_3)=1\\
 \max \{|x_1|,z_2z_3z_4\} \times \max \{|x_2|,z_1z_3z_4\} \times \max \{|x_3|,z_1z_2z_4\} \leqslant B\\
 x_1z_1+x_2z_2+x_3z_3=0
 \end{array}
\right\}.
}
As explained in Section \ref{geom}, the subvariety of $\mathbb{A}^7$ given by the equation $x_1z_1+x_2z_2+x_3z_3=0$ along with the conditions
$$
\gcd(x_1,z_2z_3z_4)= \gcd(x_2,z_1z_3z_4)= \gcd(x_3,z_1z_2z_4)=1
$$
and
$$
 \gcd(z_1,z_2)=\gcd(z_1,z_3)=\gcd(z_2,z_3)=1
$$
is the versal torsor of $\widehat{W}_3$ and hence, through this parametrization, we just performed a descent on the versal torsor of $\widehat{W}_3$.

\subsection{The case of $\widetilde{W}_n$}\label{ut2}

We now want to estimate, for $B \geqslant 1$, the quantity
$$
N_{\widetilde{W}_n,\widetilde{H}}(B)=\frac{1}{4}\#\left\{(\mathbf{x},\mathbf{y})\in \mathbb{Z}^n \times \mathbb{Z}_{\neq 0}^n  \hspace{1mm}:\hspace{1mm} 
\begin{array}{l}
\widetilde{H}(\mathbf{x},\mathbf{y}) \leqslant B\\
(\mathbf{x},\mathbf{y}) \mbox{ satisfy }~\eqref{eq}\\
\gcd(x_1,\dots,x_n)=\gcd(y_1,\dots,y_n)=1
\end{array}
\right\}.
$$
Clearly, we have
$$
N_{\widetilde{W}_n,\widetilde{H}}(B)=2^{n-2}\#\left\{(\mathbf{x},\mathbf{y})\in \mathbb{Z}^n \times \mathbb{Z}_{>0}^n  \hspace{1mm}:\hspace{1mm} 
\begin{array}{l}
\widetilde{H}(\mathbf{x},\mathbf{y}) \leqslant B\\
(\mathbf{x},\mathbf{y}) \mbox{ satisfy }~\eqref{eq}\\
\gcd(x_1,\dots,x_n)=\gcd(y_1,\dots,y_n)=1
\end{array}
\right\}.
$$
We can still rewrite the equation~\eqref{eq} as (\ref{torsor}) using Lemma \ref{lemmefacto} but we can no longer deduce that $z_{2^{j-1}}=1$. We only have the divisibility relation $z_{2^{j-1}} \mid x_j$. However, we have $z_N=1$.\\
\indent
Finally, one gets
$$
N_{\widetilde{W}_3,\widetilde{H}}(B)=2\,\#\left\{(\mathbf{x},\mathbf{z})\in \mathbb{Z}^n \times \mathbb{Z}_{>0}^{N-1}  \hspace{1mm}:\hspace{1mm} 
\begin{array}{l}
\displaystyle\max_{1\leqslant i \leqslant n}|x_i|^{n-1}\max_{1\leqslant i \leqslant n}\left|\prod_{1 \leqslant h \leqslant N-1} z_h^{\epsilon_i(h)}\right| \leqslant B\\[2mm]
\displaystyle\sum_{i=1}^n x_id_i=0\\[2mm]
\gcd(x_1,\dots,x_n)=1, \hspace{1mm} (z_h)_{1 \leqslant h \leqslant N-1} \mbox{ reduced}
\end{array}
\right\}
$$
and particularly, in the case $n=3$, we obtain
\begin{align}
N_{\widetilde{W}_3,\widetilde{H}}(B)&=2\#\left\{(\mathbf{x},\mathbf{z})\in \mathbb{Z}^3 \times \mathbb{Z}_{>0}^6  \hspace{1mm}:\hspace{1mm} 
\begin{array}{l}
\displaystyle\max_{1\leqslant i \leqslant 3}|x_i|^{2}\max\left\{z_1z_3z_5,z_2z_3z_6,z_4z_5z_6\right\} \leqslant B\\[2mm]
x_1z_2z_4z_6+x_2z_1z_4z_5+x_3z_1z_2z_3=0\\[2mm]
\gcd(x_1,x_2,x_3)=1, \hspace{1mm} (z_h)_{1 \leqslant h \leqslant 6} \mbox{ reduced}
\end{array}
\right\}\notag\\
&=2\#\left\{(\mathbf{x}',\mathbf{z})\in \mathbb{Z}^3 \times \mathbb{Z}_{>0}^6  \hspace{1mm}:\hspace{1mm} 
\begin{array}{l}
\displaystyle\max\left\{z_{1}|x'_1|,z_2|x'_2|,z_4|x'_3|\right\}^{2}\max\left\{z_1z_3z_5,z_2z_3z_6,z_4z_5z_6\right\} \leqslant B\\[2mm]
x'_1z_6+x'_2z_5+x'_3z_3=0\\[2mm]
\gcd(z_1x'_1,z_2x'_2,z_4x'_3)=1, \hspace{1mm} (z_h)_{1 \leqslant h \leqslant 6} \mbox{ reduced}
\end{array}
\right\}.\label{ssf}
\end{align}
It is easily seen that the coprimality conditions given by
$
\gcd(z_1x'_1,z_2x'_2,z_4x'_3)=1,$ and $(z_h)_{1 \leqslant h \leqslant 6}$ reduced are equivalent to
$$
\gcd(x'_1,x'_2,x'_3)=\gcd(x'_1,x'_2,z_3)=\gcd(x'_1,z_5,x'_3)=\gcd(z_6,x'_2,x'_3)=1
$$
together with the fact that $(z_h)_{1 \leqslant h \leqslant 6}$ is reduced. It then follows from \cite{BB} that the subvariety of $\mathbb{A}^9$ given by the equation $x'_1z_6+x'_2z_5+x'_3z_3=0$ along with the conditions
$$
\gcd(x'_1,x'_2,x'_3)=\gcd(x'_1,x'_2,z_3)=\gcd(x'_1,z_5,x'_3)=\gcd(z_6,x'_2,x'_3)=1
$$
and
$$
\begin{aligned}
\gcd(z_1,z_2)&=\gcd(z_1,z_4)=\gcd(z_1,z_6)=\gcd(z_2,z_4)=\gcd(z_2,z_5)\\
&=\gcd(z_3,z_4)=\gcd(z_3,z_5)=\gcd(z_3,z_6)=\gcd(z_5,z_6)=1\\
\end{aligned}
$$
is the versal torsor of $\widetilde{W}_3$ and hence, through this parametrization, we just performed a descent on the versal torsor of $\widetilde{W}_3$.

\section{Geometry and the constant in the case $\mathbb{P}^1 \times \mathbb{P}^1 \times \mathbb{P}^1$}\label{geom}

We give in this section the proof of the Theorem \ref{3}. For example, \cite{Tim} yields that the surface $S \subseteq \mathbb{P}^6$ cut out by the following 9 quadrics
$$
\begin{aligned}
X_1^2-X_2X_4&=X_1X_5-X_3X_4=X_1X_3-X_2X_5=X_1X_6-X_3X_5\\
&=X_2X_6-X_3^2=X_4X_6-X_5^2=X_1^2-X_1X_4+X_5X_7\\
&=X_1^2-X_1X_2-X_3X_7=X_1X_3-X_1X_5+X_6X_7=0
\end{aligned}
$$
is a del Pezzo surface of degree 6 of singularity type $\mathbf{A}_1$ with three lines, the lines being given by
$$
X_1=X_2=X_3=X_5=X_6=0, \quad X_1=X_3=X_4=X_5=X_6=0
$$
and
$$
X_3=X_5=X_6=X_1-X_4=X_1-X_2=0.
$$
The maps $f:\widetilde{W}_3 \rightarrow S$ given by
$$
\left\{
\begin{array}{l}
X_1=-y_3x_1x_2\\
X_2=-x_1(x_2y_3+x_3y_2)\\
X_3=-y_2y_3x_1\\
X_4=-x_2(x_1y_3+x_3y_1)\\
X_5=y_1y_3x_2\\
X_6=y_1y_2y_3\\
X_7=x_1x_2x_3\\
\end{array}
\right.
$$
and $g:S \rightarrow \widetilde{W}_3$ given by
$$
g([X_1:\cdots:X_7])=\left([X_1:-X_5],[X_5:X_6],[X_7:-X_1]\right)
$$
are well defined and inverse from each other. Thus $\widetilde{W}_3 \cong S$ and is therefore a del Pezzo surface of degree 6 of singularity type $\mathbf{A}_1$ with three lines, the lines being given by $y_i=y_j=0$ for $1 \leqslant i \neq j \leqslant 3$. As mentioned in the introduction, it follows then from \cite{Dan} and from this isomorphism that $\widehat{W}_3$ is an equivariant compactification of $\mathbb{G}_a^2$ and Theorem \ref{2} can be derived from the more general work of Chambert-Loir and Tschinkel \cite{CLT}. However, the method presented here using a descent on the versal torsor is different from the method in \cite{CLT} and it is always interesting to unravel a different proof.\\
\indent
Let us denote by $\widetilde{W}_3^{\ast}$ the minimal desingularisation of $\widetilde{W}_3$. The fact that the variety $O \subseteq \mathbb{A}^7$ given by
$$
 x_1z_1+x_2z_2+x_3z_3=0
$$
with the coprimality conditions 
$$
\gcd(x_1,z_2z_3z_4)= \gcd(x_2,z_1z_3z_4)= \gcd(x_3,z_1z_2z_4)=1
$$
and
$$
 \gcd(z_1,z_2)=\gcd(z_1,z_3)=\gcd(z_2,z_3)=1
$$
is the versal torsor of $\widetilde{W}_3^{\ast}$ is a consequence of work of Derenthal \cite{De}.\\
\indent
To conclude, let us briefly justify why the leading constant of Theorem \ref{2}
$$
\frac{1}{144}\left(\pi^2+24\log(2)-3\right)\prod_p \left( 1-\frac{1}{p}\right)^4\left(1+\frac{4}{p}+\frac{1}{p^2}\right)
$$
agrees with Peyre's prediction.\\
\indent
First of all, the variety $\widetilde{W}_3$ being rational, we know that $\beta(\widetilde{W}_3^{\ast})=1$ and work from Derenthal \cite{De} immediately yields $\alpha(\widetilde{W}_3^{\ast})=\frac{1}{144}$. We now have that
$$
\omega_H\left(\widetilde{W}_3^{\ast}(\mathbb{A}_{\mathbb{Q}})\right)=\omega_{\infty}\prod_p\omega_p
$$
with $\omega_p$ and $\omega_{\infty}$ being respectively the $p$-adic and archimedean densities. It is now easy to get that
$$
\omega_p=\frac{\#O(\mathbb{F}_p)}{p^6}=\left( 1-\frac{1}{p}\right)^4\left(1+\frac{4}{p}+\frac{1}{p^2}\right)
$$
either by direct computation or by calling upon a more general result of Loughran \cite{Dan}. Turning to the archimedean density and reasoning like in \cite{Blomer2014} one gets that
$
\omega_{\infty}
$
is given by the archimedean density on the open subset $y_1\neq0$, $y_2\neq 0$ and $y_3 \neq 0$ of $\widetilde{W}_3$. This is the affine variety given by the equation
$$
u_1+u_2+u_3=0.
$$
Using a Leray form to parametrize in $u_3$, one finally obtains
$$
\omega_{\infty}=\int_{-\infty}^{+\infty}\int_{-\infty}^{+\infty}\frac{\mbox{d}u_1\mbox{d}u_2}{\max(|u_1|,1)\max(|u_2|,1)\max(|u_1+u_2|,1)}.
$$
An easy computation now yields 
$$
\int_{-\infty}^{+\infty}\int_{-\infty}^{+\infty}\frac{\mbox{d}u_1\mbox{d}u_2}{\max(|u_1|,1)\max(|u_2|,1)\max(|u_1+u_2|,1)}=\pi^2+24\log(2)-3
$$
which finally shows that the conjectural value of Peyre's constant is
$$
\alpha(\widetilde{W}_3^{\ast})\beta(\widetilde{W}_3^{\ast})\omega_H\left(\widetilde{W}_3^{\ast}(\mathbb{A}_{\mathbb{Q}})\right)=\frac{1}{144}\left(\pi^2+24\log(2)-3\right)\prod_p \left( 1-\frac{1}{p}\right)^4\left(1+\frac{4}{p}+\frac{1}{p^2}\right)
$$
and hence that the leading constant in Theorem \ref{2} agrees with Peyre's prediction.

\section{The parabolic Eisenstein series and smooth approximations}\label{parab}
We quote the following Lemma from~\cite{Sandro2}. 

\begin{lemma}\label{ml}
Let
\es{\label{A}
\mathcal A_{\ba}(\balpha,\bbeta):=\frac18\sum_{\substack{n_1,n_2,n_3,m_1,m_2,m_3\in\Z_{\neq0},\\ a_1n_1m_1+a_2n_2m_2+a_3n_3m_3=0}}\frac{1}{|n_1|^{\alpha_1}|m_1|^{\beta_1}|n_2|^{\alpha_2}|m_2|^{\beta_2}|n_3|^{\alpha_3}|m_3|^{\beta_3}},
}
where $\ba=(a_1,a_2,a_3)\in\Z_{\neq0}^3$ and $\balpha=(\alpha_1,\alpha_2,\alpha_3),\bbeta=(\beta_1,\beta_2,\beta_3)\in\C^3$. Then $\mathcal A_{\ba}(\balpha,\bbeta)$ converges absolutely if $\Re(\alpha_i),\Re(\beta_i)>\frac23$ for all $i=1,2,3$. Moreover for $\frac23+\eps<\Re(\alpha_i),\Re(\beta_i)\leqslant \frac{11}{12}$ it satisfies $\mathcal A_{\ba}(\balpha,\bbeta)\ll1$ and
\es{\label{decomposition}
\mathcal A_{\ba}(\balpha,\bbeta)=\mathcal M_{\ba}(\balpha,\bbeta)+\mathcal E_{\ba}(\balpha,\bbeta)
}
where
\est{
&\mathcal M_{\ba}(\balpha,\bbeta)= \sum_{\substack{\{\alpha_i^*,\beta_i^*\}=\{\alpha_i,\beta_i\}\\ \forall i\in\{1,2,3\}}}\frac{2\sqrt \pi\, S_{\ba}(\balpha^*,\bbeta^*)}{\alpha_1^*+\alpha_2^*+\alpha_3^*-2}\bigg(\prod_{i=1}^3\frac{\zeta(1-\alpha_i^*+\beta_i^*)}{|a_i|^{-\alpha_i^*+\frac{1+\alpha_1^*+\alpha_2^*+\alpha_3^*}{3}}}\frac{\Gamma(-\frac{\alpha_i^*}2+\frac{1+\alpha_1^*+\alpha_2^*+\alpha_3^*}{6})}{\Gamma(\frac{1+\alpha_i^*}2-\frac{1+\alpha_1^*+\alpha_2^*+\alpha_3^*}{6})}\bigg),
}
with 
\est{
S_{\ba}(\balpha^*,\bbeta^*):=\sum_{ \ell\geqslant1}\frac{(a_1,\ell)^{1-\alpha_1^*+\beta_1^*}(a_2,\ell)^{1-\alpha_1^*+\beta_2^*}(a_3,\ell)^{1-\alpha_3^*+\beta_3^*}}{\ell^{3-\sum_{i=1}^3(\alpha_i^*-\beta_i^*)}}\varphi(\ell),
}
and where $\mathcal E_{\ba}(\balpha,\bbeta)$ is an holomorphic function on
\es{\label{hre}
\Omega_\eps:=\left\{(\balpha,\bbeta)\in\C^6\mid  \Re(\alpha_i),\Re(\beta_i)\in[\tfrac5{12}+\eps,\tfrac{11}{12}-\eps]\ \forall i \in \{1,2,3\},\ \eta<\tfrac29-\eps\right\}
}
for all $\eps>0$ with $\eta:=\sum_{i=1}^3(|\Re(\alpha_i)-\tfrac23|+|\Re(\beta_i)-\tfrac23|)$. Moreover, for 
$(\balpha,\bbeta) \in\Omega_\eps$ one has
\es{\label{ml_bound}
\mathcal E_{\ba}(\balpha,\bbeta)\ll\bigg(\big(\max_{1 \leqslant i\leqslant 3}|a_i|\big)^{14}\Big(1+\max_{1 \leqslant i\leqslant 3}(|\Im(\alpha_i)|+|\Im(\beta_i)|)\Big)^{21}\bigg)^{\frac{9\eta+18\eps}{4-9\eta}}.
}
\end{lemma}
\begin{proof}
This is the case of $k=3$ of Lemma~4 and Theorem~3 (in the form of Remark~2) of~\cite{Sandro2}  where the parameters $(s,\balpha,\bbeta,\eps)$  appearing there are taken to be $(\frac23,\balpha-\frac23,\bbeta-\frac23,3\eps)$. Notice that, contrary to~\cite{Sandro2}, we don't have the fraction $\frac32$ in front of the sum in the definition of $\eta$.
\end{proof}

{Since Lemma~\ref{ml} constitutes the main tool for our proof of Theorems~\ref{mtt1} and~\ref{2}, we say a few words about its proof. First, one divides the variables $r_i=n_im_i$ in various ranges and eliminates the largest one (say $r_1$) using the linear relation among them. In order to do this one has to write $\sum_{n_1m_1=r_1}|n_1|^{-\alpha_1}|m_1|^{-\beta_1}$ in an efficient way in terms of the remaining variables. This is done by using (a shifted version of) the identity of Ramanujan for the divisor function $\tau$ in terms of Ramanujan sums, in combination with a careful use of Mellin transforms to separate variables in expressions such as $(r_2\pm r_3)^{s}$. After the variables are completely separated, one applies Voronoi's summation formula to the sums over $r_2$ and $r_3$. The main terms will then give the polar structure, whereas the error term will produce functions which are holomorphic on the stated range.
}
\newline
\newline
\indent
The following lemma allows us to replace the characteristic function of the interval $[0,1]$ by a smooth approximation at a cost of a controlled error.

\begin{lemma}\label{cut-off}
Let $f(x)=e^{-1/(x-x^2)}$ for $0< x< 1$ and $f(x)=0$ otherwise. Let $C:=\int_0^1f(y){\rm{d}}y $ and for $0<\delta<1/2$, let
 $$F_\delta^{\pm}(x):=\frac1{\delta C}\int_x^{+\infty} f\left(\frac{y-1+(1\mp1)\delta/2}\delta\right)\,{\rm{d}}y, \quad x \in \mathbb{R}^+.$$ 
 Then $F_\delta^{\pm}\in\mathcal C^{\infty}\left(\R^+\right)$, $F_\delta^{\pm}(x)=1$ for $x\leqslant 1-\delta$, $F_{\delta}^\pm(x)=0$ for $x\geqslant 1+\delta$ and, for $x \geqslant 0$,
 $$0\leqslant F_{\delta}^-(x)\leqslant \chi_{[0,1]}(x)\leqslant F_\delta^{+}(x),$$ 
 where $\chi_{[0,1]}$ is the indicator function of the interval $[0,1].$ Moreover, the Mellin transform $\tilde F_\delta^{\pm}(s)$ of $F_\delta^{\pm} (x)$ is holomorphic in $\C\setminus\{0\}$ with a simple pole of residue one at $s=0$ and  for all $n\geqslant 0$ it satisfies for all $s \in \C\setminus\{0\}$
 \es{\label{bounds_cut-off}
 \tilde F^\pm_{\delta}(s)\ll_n \frac1{\delta^n(1+|s|)^{n+1}}, \qquad \tilde F^\pm_{\delta}(s)-\frac1s\ll \min(\delta,(1+|s|)^{-1}).
 }
\end{lemma}
\begin{proof}
The statements on $F_\delta^{\pm}$ are immediate from the definitions. Moreover, assuming $\Re(s)>0$ and integrating by parts we have
\est{
\tilde F_\delta^{\pm}(s)=\frac1{\delta Cs}\int_0^{+\infty} x^s f\bigg(\frac{x-1+(1\mp1)\delta/2}\delta\bigg)\,\mbox{d}x=\frac1{sC}\int_0^1(1+\delta x-(1\mp1)\delta/2)^sf(x)\,\mbox{d}x,
} 
the last inequality resulting from a change of variable. This already shows that $\tilde F_\delta^{\pm}$ is holomorphic in $\C\setminus\{0\}$ with a simple pole of residue one at $s=0$. Integrating by parts $n$ times then yields $ \tilde F^\pm_{\delta}(s)\ll_n \delta^{-n}(1+|s|)^{-n-1}$, which also implies the second bound in~(\ref{bounds_cut-off}) if $\delta\geqslant 1/|s|$.
Finally, if $\delta<1/|s|$ we have
\est{
\tilde F^\pm_{\delta}(s)-\frac1s=\frac1{sC}\int_0^1\big((1+\delta x-(1\mp1)\delta/2)^s-1\big)f(x)\,\mbox{d}x\ll \delta
}
since $(1+x)^s=1+O(|sx|)$ for $|sx|<1$, $|x|\leqslant\frac 12$. 
\end{proof}

\section{Proof of Theorem \ref{2} neglecting the coprimality conditions}\label{nocop}
By Section~\ref{ut1}, we need to count the integer solutions to
\es{\label{torsor_eq}
x_1z_1+x_2z_3+x_3z_3=0
}
satisfying the inequality $\max\{|x_1|,z_2z_3z_4\}\times \max\{|x_2|,z_1z_3z_4\}\times\max\{|x_3|,z_1z_2z_4\}\leqslant B$ and the coprimality conditions
\es{\label{cop_cond}
(z_1,z_2)=(z_1,z_3)=(z_2,z_3)=1,\quad (x_1,z_2z_3z_4)=(x_2,z_1z_3z_4)=(x_3,z_2z_3z_4)=1
}
with $z_1,z_2,z_3,z_4>0$. The case where $x_1x_2x_3=0$ can be dealt with easily and we postpone its treatment to section \ref{arith}, so we focus on the case where $x_1x_2x_3\neq0$. We start with the following proposition which gives an asymptotic formula for the number of solutions to the more general equation $a_1x_1z_1+a_2x_2z_3+a_3x_3z_3=0$ without imposing any coprimality condition. These conditions do not factor out immediately at the beginning of the argument, so one cannot deduce Theorem~\ref{2} directly from the Proposition \ref{mlc}, however it is instructive to prove this result first, as all the analytic difficulties are exactly the same but the notations and the arithmetic are simplified. In Section~\ref{arith} we shall conclude the proof of Theorem~\ref{2} by performing the required arithmetic computations and indicating the minor differences in the analytic argument.

\begin{prop}\label{mlc}
Let $B\geqslant1$ and $\varepsilon>0$. Let $\ba=(a_1,a_2,a_3)\in\Z_{\neq0}^3$ and 
\es{\label{def_K}
&K_{\ba}(B):=\# \Bigg\{ (\bx,\bz)\in \Z_{\neq0}^3\times\Z^4_{>0} \Bigg|
  \begin{aligned}
  & a_1x_1z_1+a_2x_2z_2+a_3x_3z_3=0 \\ 
& \max\{|x_1|,z_2z_3z_4\}\times \max\{|x_2|,z_1z_3z_4\}\times\max\{|x_3|,z_1z_2z_4\}\leqslant B
  \end{aligned}\Bigg\}.\\
}
Then there exists a polynomial $P$ of degree $3$ such that
\es{\label{fv}
 K_{\ba}(B):=BP(\log B)+O_{\varepsilon}\Big(B^{{\frac{296}{297}}+\eps}\max_{1\leqslant i\leqslant 3}|a_i|^{{14}}\Big).
}
The polynomial  $P$ has leading coefficient $\frac1{144}\mathcal I_{\ba}{{\mathfrak S}'_\ba}$, where 
\es{\label{defI}
\mathcal I_{\ba}:=\iint_{{[-1,1]^3\times [0,1]^2}}\chi_{[0,|a_3/z|]}\Big(a_1\frac{x_1}{y_1}+a_2\frac{x_2}{y_2}\Big){\rm{d}}x_1{\rm{d}}x_2{\rm{d}}z\frac{{\rm{d}}y_1{\rm{d}}y_2}{|a_3|y_1y_2}
}
and
\est{
{{\mathfrak S}'_\ba}&:=\sum_{\ell=1}^\infty\frac{(a_1,\ell)(a_2,\ell)(a_3,\ell)\varphi(\ell)}{\ell^{3}}.
}

\end{prop}
\begin{proof}
In the set defining $K_{\ba}(B)$ we have $8$ inequalities coming from all the possible values taken by the maxima. In other words, given each subset $I\seq S_3:=\{1,2,3\}$  we have the condition $$\frac{(z_1z_2z_3z_4)^{|J|}}{B}\prod_{i\in I}|x_i|\prod_{j\in J}z_j^{-1}\leqslant 1$$ where $J:=S_3\setminus I$. Now, let $0<\delta<\frac12$ and $F^{\pm}_{\delta}$ be as in Lemma~\ref{cut-off}. Then we have $K_{\ba}^-(B)\leqslant K_{\ba}(B)\leqslant K_{\ba}^+(B)$, where 
\est{
&K_{\ba}^\pm(B):=\sum_{(\bx,\bz)\in \Z_{\neq0}^3\times\Z^4_{>0}\atop a_1x_1z_1+a_2x_2z_2+a_3x_3z_3=0} \prod_{I\seq S_3}F^{\pm}_{\delta}\Bigg(\frac{(z_1z_2z_3z_4)^{|J|}}{B}\prod_{i\in I}|x_i|\prod_{j\in J}z_j^{-1}\Bigg).
}
Clearly it is sufficient to show that~\eqref{fv} holds  for both $K_{\ba}^-(B)$ and $K_{\ba}^+(B)$ with the same polynomial~$P$.
We now write each $F_{\delta}^{\pm}$ in terms of its Mellin transform using the variable $s_{I}$ for the cut-off function corresponding to the set $I$. For brevity we shall often indicate for example with $s_{123}$ the variable $s_{\{1,2,3\}}$ and with $c_{\{1,2,3\}}$ or $c_{123}$ the corresponding line of integration, and similarly for the other variables. In particular, we will denote by $s$ the variable $s_{\varnothing}$.
 As lines of integration, we take  $c_{I}=\frac{|I|}{12}+\eps$ for all $I$ for a fixed $\varepsilon>0$ small enough. Doing so we obtain
\est{
K_{\ba}^\pm(B)=\!\!\!\hspace{-2em}\sum_{(\bx,\bz)\in \Z_{\neq0}^3\times\Z^4_{>0}\atop a_1x_1z_1+a_2x_2z_2+a_3x_3z_3=0}\frac1{(2\pi i)^8}\iint_{(c_{I})}\frac{B^{\sum_{I}s_{I}}}{z_4^{\sum_{I}s_I(3-|I|)}}\prod_{i=1}^3|x_i|^{-\sum_{I, i\in I}s_I}z_i^{-\sum_{I}s_{I}(2+\delta_{i\in I}-|I|)}
\prod_I\,\tilde F_\delta^\pm (s_I)\mbox{d}s_I
}
where $\delta_{i\in I}=1$ if $i\in I$ and $\delta_{i\in I}=0$ otherwise and where the sums inside the integrals are over $I\subseteq S_3$. Notice that with this choice we have
\est{
\sum_{I\subseteq S_3}c_I=1+8\eps,\quad \sum_{I\subseteq S_3\atop i\in I}c_I=\frac23+4\eps,\quad \sum_{I\subseteq S_3}c_{I}(2+\delta_{i\in I}-|I|)=\frac23+8\eps,\quad \sum_{I\subseteq S_3}c_I(3-|I|)=1+12\eps.
}
In particular the above series are absolutely convergent by Lemma~\ref{ml}. Now, write 
\begin{equation}
\begin{aligned}
\xi&:=\tfrac12(2s+s_1+s_2+s_3-s_{123})\\
\alpha_1&:=\tfrac12(2 s +3 s_1+s_2 + s_3 + 2 s_{12} + 2 s_{13}+  s_{123}  )=\sum_{1\in I}s_I+\xi=\sum_{I}s_{I}(2+\delta_{1\in I}-|I|)-\xi\\
\alpha_2&:=\tfrac12(2 s + s_1+3s_2 + s_3 + 2 s_{12} + 2 s_{23}+  s_{123}   )=\sum_{2\in I}s_I+\xi=\sum_{I}s_{I}(2+\delta_{2\in I}-|I|)-\xi\\
\alpha_3&:=\tfrac12(2 s + s_1+s_2 +3 s_3 + 2 s_{13} + 2 s_{23}+  s_{123}   )=\sum_{3\in I}s_I+\xi=\sum_{I}s_{I}(2+\delta_{3\in I}-|I|)-\xi\\
\end{aligned}
\label{nn}
\end{equation}
where we are neglecting here the dependencies on the variables $s_I$ in the notations in order to simplify the exposition.
 Notice then that 
\est{
\sum_{I\seq S_3}s_I=\tfrac12(\alpha_1+\alpha_2+\alpha_3-\xi),\qquad \sum_{I\seq S_3}s_I(3-|I|)=\tfrac12(\alpha_1+\alpha_2+\alpha_3+3\xi).
}
Thus, summing the Dirichlet series we have
\est{
K_{\ba}^\pm(B)=\frac1{(2\pi i)^8}\iint_{(c_{I})}B^{\frac12(\alpha_1+\alpha_2+\alpha_3-\xi)}\zeta(\tfrac12(\alpha_1+\alpha_2+\alpha_3+3\xi))\mathcal A_{\ba}(\balpha-\xi,\balpha+\xi)
\prod_I\,\tilde F_\delta^\pm (s_I)\mbox{d}s_I
}
with the notation of Lemma \ref{ml}. By Lemma~\ref{ml} and using the notations (\ref{decomposition}), we can split $\mathcal A_\ba(\balpha-\xi,\balpha+\xi)$ into $$\mathcal A_\ba(\balpha-\xi,\balpha+\xi)=\mathcal {M}_\ba(\balpha-\xi,\balpha+\xi)+\mathcal {E}_\ba(\balpha-\xi,\balpha+\xi)$$ thus obtaining the corresponding decomposition $K_{\ba}^\pm(B)=M_{\ba}^\pm(B)+E_{\ba}^\pm(B)$, with
\es{\label{M}
M_{\ba}^\pm(B):=\frac1{(2\pi i)^8}\iint_{(c_{I})}B^{\frac12(\alpha_1+\alpha_2+\alpha_3-\xi)}\zeta(\tfrac12(\alpha_1+\alpha_2+\alpha_3+3\xi))\mathcal M_{\ba}(\balpha-\xi,\balpha+\xi)
\prod_I\,\tilde F_\delta^\pm (s_I)\mbox{d}s_I
}
and
\est{
E_{\ba}^\pm(B):=\frac1{(2\pi i)^8}\iint_{(c_{I})}B^{\frac12(\alpha_1+\alpha_2+\alpha_3-\xi)}\zeta(\tfrac12(\alpha_1+\alpha_2+\alpha_3+3\xi))\mathcal E_{\ba}(\balpha-\xi,\balpha+\xi)
\prod_I\,\tilde F_\delta^\pm (s_I)\mbox{d}s_I.
}
In the latter integral we move the line of integration $c_{S_3}$ to $c_{S_3}=\frac{1}{4}-\frac{2}{27}+6\eps=\frac{19}{108}+6\eps$. Notice that doing so, in the new lines of integration, we have $\Re(\alpha_i+\xi)=\frac23+8\eps$ and $\Re(\alpha_i-\xi)=\frac{2}3-\frac2{27}+9\eps=\frac{16}{27}+9\eps$ for all $i\in \{1,2,3\}$. In particular we stay on the right of the pole of the $\zeta$-function and we  have $\eta=\sum_{i=1}^3(|\Re(\alpha_i-\xi)-\tfrac23|+|\Re(\alpha_i+\xi)-\tfrac23|)=\frac29-3\eps$. By~\eqref{ml_bound} we then have for $\eps$ small enough,
\est{
E_{\ba}^\pm(B)\ll A^{{14}}B^{\frac{25}{27}+13\eps}\iint_{(c_{I})}\big(1+\max_{I}|s_I|\big)^{{21}}
\prod_I\, \big|\tilde F_\delta^\pm (s_I)\big| \mbox{d}s_I
}
where $\displaystyle A:=\max_{1\leqslant i\leqslant 3}|a_i|$. Now, we have
\begin{equation}
\begin{aligned}
&\int_{1}^{+\infty}\min\Big(\frac1{x},\frac1{\delta x^2}\Big)\mbox{d}x\ll |\log \delta|\ll_{\varepsilon} \delta^{-\varepsilon},\\
& \int_{1}^{+\infty}x^{{21}}\min\Big(\frac1{x},\frac1{{\delta^{22} x^{23}}}\Big)\mbox{d}x\leqslant \int_{1}^{1/\delta}x^{{20}}\,\mbox{d}x+\delta^{-{22}}\int_{1/\delta}^{+\infty} x^{-{2}}\,\mbox{d}x\ll \delta^{-{21}}
\end{aligned}
\label{bb}
\end{equation}
and so, using Lemma~\ref{cut-off} we find
\es{\label{dax}
E_{\ba}^\pm(B)\ll A^{{14}}B^{\frac{25}{27}+13\eps}\delta^{-{21}-\eps}.
}
Now, we consider the main term $M_{\ba}^\pm(B)$ defined in (\ref{M}). We can write 
 $\mathcal {M}_{\ba}(\balpha-\xi,\balpha+\xi)$ as
 \est{
\mathcal {M}_\ba(\balpha-\xi,\balpha+\xi)=\sum_{k=0}^3\frac{\zeta(1+2\xi)^k\zeta(1-2\xi)^{3-k}}{\alpha_1+\alpha_2+\alpha_3+(3-2k)\xi-2} \mathcal {Q}_{\ba,k}({\balpha,\xi})
}
with
\es{\label{defQ}
\mathcal {Q}_{\ba,k}({\balpha,\xi})&:=\sum_{\bepsilon\in\{\pm1\}^3\atop \#\{i\mid \epsilon_i=1\}=k}\sum_{\ell=1}^{+\infty}\frac{(a_1,\ell)^{1+2\epsilon_1\xi}(a_2,\ell)^{1+2\epsilon_2\xi}(a_3,\ell)^{1+2\epsilon_3\xi}}{\ell^{3+2(2k-3)\delta}}\varphi(\ell)\\
&\quad\times
2\pi^\frac12\prod_{i=1}^3\frac{\Gamma(\frac{-\alpha_i+\epsilon_i\xi}2+\frac{1+\alpha_1+\alpha_2+\alpha_3+(3-2k)\xi}{6})}{|a_i|^{-\alpha_i+\epsilon_i\xi+\frac{1+\alpha_1+\alpha_2+\alpha_3+(3-2k)\xi}{3}}\Gamma(\frac{1+\alpha_i-\epsilon_i\xi}2-\frac{1+\alpha_1+\alpha_2+\alpha_3+(3-2k)\xi}{6})}
}
and where the sum is over  $\bepsilon=(\epsilon_1,\epsilon_2,\epsilon_3)\in\{\pm1\}^3$. Notice that in the region
\es{\label{region_Q_bounded}
{0\leqslant}{\Re\bigg({-\alpha_i+\epsilon_i\xi}+\frac{1+\alpha_1+\alpha_2+\alpha_3+(3-2k)\xi}{3}\bigg)\leqslant \frac12\quad \forall i\in\{1,2,3\}},\qquad {|\Re(\xi)|\leqslant\tfrac16-\eps}
}
we have that $\mathcal {Q}_{\ba,k}(\balpha,\xi)$ is holomorphic and satisfies
\es{\label{boundQ}
\mathcal {Q}_{\ba,k}(\balpha,\xi)\ll {A^{3+6\Re(\xi)}}
}
uniformly in $\ba$, by the bound $\left|(a,\ell)^s/\ell^s\right|\leqslant {(|a|/\ell)^{\Re(s)}}$ {for $\Re(s)\geq0$}, and since
\es{\label{sfg}
\Gamma(\tfrac s2)/\Gamma(\tfrac{1-s}2)\ll_\sigma (1+|t|)^{\sigma-\frac12}
}
by Stirling's formula~\cite[(8.328.1)]{GR}.

Then, we write $M_{\ba}^\pm(B)=\displaystyle \sum_{k=0}^3M_{\ba,k}^\pm(B)$ where
\est{
M_{\ba,k}^\pm(B)&=\frac1{(2\pi i)^8}\iint_{(c_{I})}B^{\frac12(\alpha_1+\alpha_2+\alpha_3-\xi)}\zeta(\tfrac12(\alpha_1+\alpha_2+\alpha_3+3\xi))\frac{\zeta(1+2\xi)^k\zeta(1-2\xi)^{3-k}}{\alpha_1+\alpha_2+\alpha_3+(3-2k)\xi-2} \\
&\hspace{7em}\times\mathcal {Q}_{\ba,k}({\balpha,\xi})
\prod_I\,\tilde F_\delta^\pm (s_I)\mbox{d}s_I
}
and with all the lines of integration still at $c_{I}=\frac{|I|}{12}+\eps$ for all $I$.
If $k\in\{0,1\}$ we move the line of integration $c_{S_3}$ and $c$ to $c_{S_3}=\frac14 -\frac 8{81}+\eps$ and $c=\frac2{81}+\eps$ without passing through any pole. Indeed, $\Re(\xi)=\frac2{27}+2\eps$ stays positive, $\Re(\alpha_1+\alpha_2+\alpha_3+\xi)=2+20\eps$ is unchanged and in particular  $\Re(\alpha_1+\alpha_2+\alpha_3+3\xi)>2$. Thus, for $k\in \{0,1\}$ we have
\est{
M_{\ba,k}^\pm(B)
&\ll {A^{\frac{31}9+12\eps}}B^{\frac{25}{27}+8\eps}\iint_{(c_{I})}(1+\max_I|s_I|) \prod_I\,\big|\tilde F_\delta^\pm (s_I)\big|\mbox{d}s_I\ll {A^{\frac{31}9+12\eps}}B^{\frac{25}{27}+8\eps}\delta^{-1-\eps}
}
since we are inside the region~\eqref{region_Q_bounded} and since $|\zeta(1-2\xi)|^3\ll 1+|\xi|$ for $\Re(\xi)=\frac2{27}+\eps$ by the convexity bound~\cite[(5.1.4)]{Tit}.

Now, consider the case $k\in\{2,3\}$. In those cases and for $\varepsilon$ small enough we move the line of integration $c_{S_3}$ to $c_{S_3}=\frac{19}{108}+\eps$ passing through the pole at $\alpha_1+\alpha_2+\alpha_3+(3-2k)\xi-2=0$, namely $s_{123} =\frac2k-\frac1k \sum_{I\neq S_3}((2-|I|)(3-k)+|I|)s_I$ with respect to~$s_{123}$, but without crossing the poles of the~$\zeta$ functions since we increased $\Re(\xi)$ from $2\varepsilon$ to $\frac1{27}+2\varepsilon$. The contribution of the integral on the new lines of integration is easily seen to be $O\left({A^{\frac{29}9+12\eps}}B^{\frac{25}{27}+8\eps}\delta^{-1-\eps}\right)$ and so we are left with examining the contribution of the residue. 

First we consider the case $k=2$. As mentioned above, with respect to $s_{123}$ the pole is located at $s_{123} =1- \sum_{I\neq S_3}s_I$. Also, we can replace each $\tilde F_\delta^\pm (s_I)$ with $\frac1{s_I}$ at a cost of committing an error which, by~\eqref{bounds_cut-off} and~\eqref{boundQ} since~\eqref{region_Q_bounded} is satisfied, is bounded by
\est{
&\ll B\max _{I'\subseteq S_3} \iint_{(c_{I})}{A^{3+6\Re(\xi)}}|\zeta(1-2\xi)|\min\left(\delta,\frac1{|s_{I'}|}\right)
\prod_{I\neq I'}\frac1{|s_I|}\cdot\prod_{I\neq S_3}\mbox{d}s_I
}
where $s_{123} :=1- \sum_{I\neq S_3}s_I$. In particular $\Re(s_{123})=\frac14-7\eps$. Also, $\Re(\xi)=6\eps$ and so $|\zeta(1-2\xi)|\ll |\xi|^{7\eps}$ by the convexity bound~\cite[(5.1.4)]{Tit}. Thus, the above is, for $I' \neq S_3$,
\es{\label{decb}
&\ll  \iint_{\substack{(c_{I})\\ |s_{I'}|\leqslant\delta^{-1}}}\frac{\delta {A^{3+36\eps}}B \max_{I}|s_I|^{7\eps}}{|1-\sum_{I\neq S_3}s_I|}
\prod_{I\neq I' ,S_3}\frac1{|s_I|}\cdot\prod_{I\neq S_3}\mbox{d}s_I+ \iint_{\substack{(c_{I})\\ |s_{I'}|\geqslant\delta^{-1}}}\frac{{A^{3+36\eps}}B\max_{I}|s_I|^{7\eps}}{|1-\sum_{I\neq S_3}s_I|}
\prod_{I\neq S_3}\frac{\mbox{d}s_I}{|s_I|}\\
&\ll {A^{3+36\eps}} B\delta^{1-7\eps}
}
{and a similar argument gives the same bound also for $I=S_3$}.
It follows that 
\est{
M_{\ba,2}^\pm(B)=W_\ba B+O\left({A^{\frac{31}9+12\eps}}B^{\frac{25}{27}+8\eps}\delta^{-1-\eps}+{A^{3+36\eps}}B\delta^{1-8\eps}\right)
}
where
\est{
W_\ba:=\frac1{(2\pi i)^7}\iint_{(c_{I})}\frac{\zeta(1+2\xi)^3\zeta(1-2\xi)}{2\big(1- \sum_{I\neq S_3}s_I\big)} \mathcal {Q}_{\ba,k}({\balpha,\xi})
\prod_{I\neq S_3}\frac{\mbox{d}s_I}{s_I}
}
and where $\boldsymbol{\alpha}$ and $\xi$ are given by (\ref{nn}) with $s_{123}$ replaced by $1- \sum_{I\neq S_3}s_I$.\\
\indent
Now, let us consider the case $k=3$. We proceed as above replacing $\tilde F_\delta^\pm (s_I)$ by $s_I^{-1}$ for all $I\neq \emptyset$. We can't do the same for $I=\emptyset$ yet because the pole giving the residue is, for $\eps$ small enough, at $s_{123}=\frac23-\frac13\sum_{I\neq S_3}|I|s_I$ with respect to $s_{123}$ which does not depend on $s$ and thus the integral with $\tilde F_\delta^\pm (s)$ replaced by $1/s$ is not absolutely convergent. We arrive to
\est{
M_{\ba,3}^\pm(B)&=\frac1{(2\pi i)^7}\iint_{(c_{I})}B^{1+\xi}\frac{\zeta(1+3\xi)\zeta(1+2\xi)^3}{2-\sum_{I\neq S_3}|I|s_I } \mathcal {Q}_{\ba,3}({\balpha,\xi})\,
s\,\tilde F_\delta^\pm (s)\prod_{I\neq S_3}\frac{\mbox{d}s_I}{s_I}\\
&\quad+O\left({A^{\frac{31}9+12\eps}}B^{\frac{25}{27}+8\eps}\delta^{-1-\eps}+{A^{3+36\eps}}B^{1+4\eps}\delta^{1-8\eps}\right),
}
where the lines of integration are still at $c_{I}=\frac{|I|}{12}+\eps$ for all $I\neq S_3$. Next, we move the lines of integration~$c_{I}$ to $c_{I}=\eps$ for all $I$ satisfying $|I|=1$. This has the effect of moving $\Re(\xi)$ from  $4\eps$ to $-\frac16+4\eps$ and $\Re(\frac23-\frac13\sum_{I\neq S_3}|I|s_I)$ from $\frac14-3\eps$ to $\frac13-3\eps$. In particular we stay on the right of the poles at $s_{I}=0$ for all $I$ and we encounter a quadruple pole at $s:=\frac13-\frac13\sum_{I \neq \emptyset,S_3}(3-|I|)s_I$ with respect to $s$. Note that we have $\xi=0$ at the quadruple pole. The contribution of the integrals on the new lines of integration is, as in~\eqref{bb}
\est{
\ll
{A^{2+24\eps}}B^{\frac56+4\eps}\iint_{(c_{I})}\frac{\max_I|s_I|^\frac34}{\big|2-\sum_{I\neq S_3}|I|s_I\big|} \,
|s\,\tilde F_\delta^\pm (s)|\prod_{I\neq S_3}\frac{\mbox{d}s_I}{|s_I|}\ll {A^{2+24\eps}}B^{\frac56+4\eps}\delta^{-\frac34-\eps},\\
}
since we are on the region~\eqref{region_Q_bounded} and since, by the convexity bound \cite[(5.1.4)]{Tit}, $|\zeta(1+3\xi)\zeta(1+2\xi)|^3\ll |\xi|^{\frac34}$ for $\Re(\xi)=-\frac16+4\eps$. As for the residue, we notice
that we can replace $s\tilde F_\delta^\pm (s)$ by $1$ at a cost of an error which is $O\left(\delta {A^{3+6\eps}}B^{1+5\eps}\right)$. Indeed, we can write the residue as an integral in $s$ along a circle of radius $\eps$ around $\frac13-\frac13\sum_{I\neq \emptyset, S_3}(3-|I|)s_I=O(\eps)$.
 We then use $s\tilde F_\delta^\pm (s)-1=O(|s|\delta)$ coming from (\ref{bounds_cut-off}) and bound trivially the integrals. Thus, we have
\est{
M_{\ba,3}^\pm(B)&=\frac1{(2\pi i)^6}\iint_{(c_{I})}\Res_{\xi=0}\Big(B^{1+\xi}\frac{\zeta(1+3\xi)\zeta(1+2\xi)^3}{3s_{123}\,s} \mathcal {Q}_{\ba,3}({\balpha,\xi})\Big)\,
\prod_{I\neq \emptyset, S_3}\frac{\mbox{d}s_I}{s_I}\\
&\quad+O\left(B^{\frac{25}{27}+8\eps}\delta^{-1-\eps}+B^{1+5\eps}\delta\right)
}
with 
\begin{equation}
s:=\frac13-\frac13\sum_{I\neq \emptyset,S_3}(3-|I|)s_I, \quad s_{123}:=\frac23-\frac13\sum_{I\neq S_3}|I|s_I,
\label{ss}
\end{equation}
$\boldsymbol{\alpha}$ given by (\ref{nn}) with $s_{123}$ and $s$ replaced by (\ref{ss})
and lines of integration which we can take to be $c_{I}=\frac1{12}$ for all $I\neq S_3,\emptyset$. Note that computing the residue in $\delta$ rather than in $s$ doesn't change the result.
Computing the residue then gives
\est{
M_{\ba,3}^\pm(B)&=BP_\ba(\log B)+O\left({A^{\frac{31}9+12\eps}}B^{\frac{25}{27}+8\eps}\delta^{-1-\eps}+{A^{3+36\eps}}B^{1+5\eps}\delta^{1-7\eps}\right),
}
where $P_\ba$ is a degree $3$ polynomial with leading coefficient
\es{\label{w3}
P_{\ba,3}&:=\frac1{(2\pi i)^6}\iint_{(\frac1{12})}\frac{\mathcal {Q}_{\ba,3}({\balpha,0})}{432\,s_{123}\,s} \,
\prod_{I\neq \emptyset, S_3}\frac{\mbox{d}s_I}{s_I}
}
with $s$ and $s_{123}$ given by (\ref{ss}), 
\es{\label{af}
\alpha_2=\tfrac13 (2 + 2 s_2 -s_1-s_3+ s_{12} + s_{23}  - 2 s_{13}),\qquad
\alpha_3=\tfrac13 (2 + 2 s_3 -s_1-s_2+ s_{13} + s_{23}  - 2 s_{12}),
}
and $\alpha_1=2-\alpha_2-\alpha_3$, where again we neglect the dependencies on $s_I$ for $I \neq \emptyset, S_3$ in the notations. Collecting the above results, we have
\est{
K_{\ba}^\pm(B)&=BP_\ba(\log B)+W_\ba B+O\left({A^{3+36\eps}}B^{1+5\eps}\delta^{1-7\eps}+A^{{14}}B^{\frac{25}{27}+13\eps}\delta^{-{21}-\eps}\right)\\
&=BP_\ba(\log B)+W_\ba B+O\left(A^{{14}}B^{{\frac{296}{297}}+14\eps}\right),
}
upon choosing $\delta={B^{-\frac{1}{297}}}$. Thus, it remains to show that $P_{\ba,3}=\frac1{144}\mathcal I_{\ba}{{\mathfrak S}'_\ba}$ with the notations of Proposition \ref{mlc}.

 First, we notice that, for $\alpha_1+\alpha_2+\alpha_3=2$ and $\xi=0$, $\mathcal {Q}_{\ba,3}(\balpha,\xi)$ simplifies to
\est{
\mathcal {Q}_{\ba,3}({\balpha,0})&=2\pi^\frac12{{\mathfrak S}'_\ba}\prod_{i=1}^3\frac{\Gamma(\frac{1-\alpha_i}2)}{|a_i|^{1-\alpha_i}\Gamma(\frac{\alpha_i}2)}.
}
Next, we use $\alpha_2,\alpha_3$ as new variables, 
writing $s_{2}$ and $s_3$ as
\begin{equation}
\begin{aligned}	
s_2=-2+2\alpha_2+\alpha_3+s_1+s_{13}-s_{23},\qquad
s_3=-2+2\alpha_3+\alpha_2+s_1+s_{12}-s_{23}.	
\end{aligned}
\label{cvv}
\end{equation}
Note that with this change of variables, we also have
\begin{equation}
\begin{aligned}	
s_{123} = 2-\alpha_2 - \alpha_3 - s_1 - s_{12} - s_{13},\qquad
s=3-2\alpha_2-2\alpha_3-2s_1-s_{12}-s_{13}+s_{23}
\end{aligned}
\label{cvv2}
\end{equation}
and remind that $\alpha_1=2-\alpha_2-\alpha_3$. The lines of integration for $\alpha_2,\alpha_3$ are at real part equal to $\frac23$. Since the Jacobian of the above change of variables is equal to $3$ we find, with (\ref{cvv}) and (\ref{cvv2})
\est{
P_{\ba,3}&=\frac1{(2\pi i)^2}\iint_{(\frac23)}\mathcal {Q}_{\ba,3}({\balpha,0})\frac1{(2\pi i)^4}
\iint_{(\frac1{12})}\frac{\mbox{d}s_1\mbox{d}s_{12}\mbox{d}s_{13}\mbox{d}s_{23}}{144\,s_{123}\,s s_2 s_3 s_1s_{12}s_{13}s_{23}} \,\mbox{d}\alpha_2\mbox{d}\alpha_3.
}
The inner integrals can be evaluated by moving each integral to $-\infty$ (or, equivalently, to $+\infty$), repeatedly applying the residue theorem. The inner integrals can be evaluated by a simple but tedious calculation, which can be readily checked by some mathematical software, whence obtaining
\est{
P_{\ba,3}
&=\frac1{(2\pi i)^2}\iint_{(\frac23)}
\frac{-\mathcal {Q}_{\ba,3}({\balpha,0})}{144\alpha_1\alpha_2\alpha_3 (\alpha_1-1)(\alpha_2-1)(\alpha_3-1)  } \,\mbox{d}\alpha_2\mbox{d}\alpha_3\\
&=
\frac{{{\mathfrak S}'_\ba}}{(2\pi i)^2}\iint_{(\frac2{3})}\frac{2\pi^\frac12}{144}
\prod_{i=1}^3\frac{\Gamma(\frac{1-\alpha_i}2)}{\alpha_i(1-\alpha_i)|a_i|^{1-\alpha_i}\Gamma(\frac{\alpha_i}2)} \,{\mbox{d}\alpha_2\mbox{d}\alpha_3}
}
with $\alpha_1=2-\alpha_2-\alpha_3$ and the result follows by the following lemma. 
\end{proof}
\begin{lemma}\label{integral}
For $\ba\in\Z_{\neq0}^3$ we have
\es{\label{altI}
\mathcal I_\ba=\frac1{(2\pi i)^2}\iint_{(\frac2{3})}2\pi^\frac12
\prod_{i=1}^3\frac{\Gamma(\frac{1-\alpha_i}2)}{\alpha_i(1-\alpha_i)|a_i|^{1-\alpha_i}\Gamma(\frac{\alpha_i}2)} {{\rm{d}}\alpha_2{\rm{d}}\alpha_3}
}
where $\alpha_1:=2-\alpha_2-\alpha_3$.
Moreover, $\mathcal I_{\bone}=\pi^2 + 24 \log 2-3$.
\end{lemma}
\begin{proof}
For  $\alpha_1+\alpha_2+\alpha_3=2$, we have the $\Gamma$ identity (see (2.8) in~\cite{Sandro2})
\est{
\pi^\frac12\prod_{i=1}^3\frac{\Gamma(\frac{1-\alpha_i}2)}{\Gamma(\frac{\alpha_i}2)}=\sum_{i=1}^3\frac{\Gamma(1-\alpha_1)\Gamma(1-\alpha_2)\Gamma(1-\alpha_3)}{\Gamma(1-\alpha_i)\Gamma(\alpha_i)}
}
and, considering $\alpha_1=2-\alpha_2-\alpha_3$ as a function of $\alpha_2$ with $\alpha_3$ fixed, we have the Mellin transforms~\cite[(7.7.9) and (7.7.14-15)]{Tit2}
\es{\label{mellin_tr}
\frac1{2\pi i}\int_{(c)}\frac{\Gamma(1-\alpha_1)\Gamma(1-\alpha_2)\Gamma(1-\alpha_3)}{\Gamma(1-\alpha_i)\Gamma(\alpha_i)} x^{\alpha_2-1}\,\mbox{d}\alpha_2 =
\begin{cases}
(1-x)^{-\alpha_3}\chi_{[0,1]}(x) & \text{if $i=1$}\\
(x-1)^{-\alpha_3}\chi_{[1,\infty)}(x) & \text{if $i=2$}\\
(1+x)^{-\alpha_3} & \text{if $i=3$}\\
\end{cases}
}
for $c>0$ and $\Re(\alpha_3)<1$ if $i=1$, $c>0$, $\Re(\alpha_3)>0$ if $i=2$ and $0<c<\Re(\alpha_3)$ if $i=3$. Also, for $0<\Re(\alpha_2),\Re(\alpha_3)<1$, we have the identity
\est{
\int_{[0,1]^4}(x_1/y_1)^{1-\alpha_2-\alpha_3}(x_2/y_2)^{\alpha_2-1}(y_1y_2)^{-1}\,\mbox{d}x_1\mbox{d}x_2\mbox{d}y_1\mbox{d}y_2=\frac1{\alpha_1(\alpha_1-1)\alpha_2(\alpha_2-1)}.
}
It follows that, {in the case $i=1$},
\est{
&\frac1{(2\pi i)^2}\iint_{(\frac23)}
\frac{\Gamma(1-\alpha_2)\Gamma(1-\alpha_3)|a_1|^{1-\alpha_2-\alpha_3}|a_2|^{\alpha_2-1}|a_3|^{\alpha_3-1}}{\alpha_1\alpha_2\alpha_3 {(1-\alpha_1)(1-\alpha_2)(1-\alpha_3)}\Gamma(\alpha_1)  } \,\mbox{d}\alpha_2\mbox{d}\alpha_3\\
&\hspace{2em} =\int_{[0,1]^4 \atop |a_1|x_1/y_1-|{a_2}|x_2/y_2\geqslant 0}\frac1{2\pi i}\int_{(\frac23)}
\frac{(|a_1|x_1/y_1-|a_2|x_2/y_2)^{-\alpha_3}|a_3|^{\alpha_3-1}}{\alpha_3{(1-\alpha_3)} } \,\mbox{d}\alpha_3(y_1y_2)^{-1}\mbox{d}x_1\mbox{d}x_2\mbox{d}y_1\mbox{d}y_2\\
&\hspace{2em}=\int_{[0,1]^5\atop {0\leqslant}|a_1|x_1/y_1-|{a_2}|x_2/y_2\leqslant |a_3|/z}(|a_3|y_1y_2)^{-1}dx_1dx_2dy_1dy_2dz,\\
}
since $\frac1{2\pi i}\int_{(\frac23)}x^{-s}\frac{ds}{s(s-1)}=\int_{0}^1\chi_{[0,1/z)}(x)dz$, for $x>0$.
One evaluates similarly the cases arising from $i\in \{2,3\}$ and~\eqref{altI} easily follows.\\
\indent
 Finally, we notice that after using the Gamma identity $\cos(\frac {\pi s}2)\Gamma(s)={\pi^{1/2}2^{s-1}\Gamma(\frac s2)}/{\Gamma(\frac{1-s}2)}$, which follows from the reflection and duplication formulae for the Gamma function, the integral on the right hand side of~\eqref{altI} reduces to the one computed in Lemma~2.10 of~\cite{BBS}, where it was shown to be equal to $\pi^2 + 24 \log 2-3$  by means of a long calculation. One could also give a shorter proof of this identity (still requiring some computations) by writing $(\alpha_1(\alpha_1-1)\alpha_2(\alpha_2-1))^{-1}$ in terms of its ($1$-variable) Mellin transform, applying~\eqref{mellin_tr} and evaluating the resulting integrals.
\end{proof}

\section{Proof of Theorem~\ref{2}}\label{arith}
We now move to the proof of Theorem \ref{2}, namely counting points satisfying~\eqref{torsor_eq} and the coprimality conditions~\eqref{cop_cond}. First, we give three lemmata which respectively remove the extra coprimality conditions by mean of M\"obius inversion formula, show the convergence of the resulting sums, and compute the Euler product arising in the main term.
\begin{lemma}\label{mobius}
Let $f:\R^7\to\C$ a function of compact support. Then,
\est{
&\sumprime_{\substack{(\bx,\bz)\in \Z_{\neq0}^3\times\Z^4_{>0}}}f(x_1,z_1,x_2,z_2,x_3,z_3,z_4)\\
&\hspace{6em}=\sum_{\be,\bl\in\N^3 \atop \bd\in\N^6}\mu(\be,\bd,\bl)\sum_{\substack{(\bx,\bz)\in \Z_{\neq0}^3\times\Z^4_{>0}}}f(b_1x_1,c_1z_1,b_2x_2,c_2z_2,b_3x_3,c_3z_3,c_4z_4),\\
}
where here and below $\sum'$ indicates that the sum is restricted to satisfy the coprimality conditions 
\est{
(z_1,z_2)=(z_1,z_3)=(z_2,z_3)=1,\quad (x_1,z_2z_3z_4)=(x_2,z_1z_3z_4)=(x_3,z_1z_2z_4)=1
}
and where $\be:=(e_1,e_2,e_3)$, $\bd:=(d_{12},d_{13},d_{21},d_{23},d_{31},d_{32})$, $\bl:=(\ell_{12},\ell_{13},\ell_{23})$,
\es{\label{bc}
&b_1:= [e_1,d_{12},d_{13}],\ b_2:=[e_2,d_{21},d_{23}],\ b_3:=[e_3,d_{31},d_{32}]\\
&c_1:=[d_{21},d_{31},\ell_{12},\ell_{13}],\ c_2:=[d_{12},d_{32},\ell_{12},\ell_{23}],\ c_3:=[d_{13},d_{23},\ell_{13},\ell_{23}],\ c_4:=[e_1,e_2,e_3]
}
and, with a slight abuse of notation,
\est{
&\mu(\be,\bd,\bl):=\mu(\ell_{12})\mu(\ell_{13})\mu(\ell_{23}) \prod_{1\leqslant i\leqslant 3}\mu(e_i)\prod_{1\leqslant i,j\leqslant 3\atop i\neq j}\mu(d_{ij}).\\
}
\end{lemma}
\begin{proof}
This is just an immediate application of M\"obius' inversion formula.
\end{proof}

\begin{lemma}\label{b_dirichlet}
With the same notations as in Lemma~\ref{mobius} and for 
\es{\label{cond_b_dirichlet}
& \Re(u_i),\Re(w_i),\Re(u_i+w_4-1-\eps)\geqslant\Re(\kappa),\ \forall i\in \{1,2,3\};\quad \Re(w_4)\geqslant0;\quad  \Re(\kappa)\geqslant 0;\\
&\Re(u_i+w_j-\kappa)>1+\eps,\ \forall i,j\in\{1,2,3\}, i\neq j;\hspace{0.75em}
\Re(w_i+w_j-\kappa)>1+\eps,\ \forall i,j\in\{1,2,3\}, i<j;
}
we have
\es{\label{bound1}
\sum_{\be,\bl\in\N^3 \atop \bd\in\N^6}\frac{\max(b_1c_1,b_2c_2,b_3c_3)^{\kappa}}{b_1^{u_1}b_2^{u_2}b_3^{u_3}c_1^{w_1}c_2^{w_2}c_3^{w_3}c_4^{w_4}}\ll 1.
}
\end{lemma}
\begin{proof}
We have the formal Euler product formula
\es{\label{e_p}
\sum_{\be,\bl\in\N^3\atop \bd\in\N^6}\frac{1}{b_1^{u_1}b_2^{u_2}b_3^{u_3}c_1^{w_1}c_2^{w_2}c_3^{w_3}c_4^{w_4}}=\prod_p\bigg(\sum_{\be',\bl'\in\Z_{\geqslant 0}^3,\bd'\in\Z_{\geqslant 0}^6}{\hspace{-0.5em}p^{-(b_1'u_1+b_2'u_2+b_3'u_3+c_1'{w_1}+c_2'{w_2}+c_3'{w_3}+c_4'{w_4})}}\bigg),
}
where 
\es{\label{euler_exp}
&b_1'=\max(e_1',d_{12}',d_{13}'),\ b_2'=\max(e_2',d_{21}',d_{23}'),\ b_3'=\max(e_3',d_{31}',d_{32}'),\ c_4'=\max(e_1',e_2',e_3')\\
&c_1'=\max(d_{21}',d_{31}',\ell_{12}',\ell_{13}'),\ c_2'=\max(d_{12}',d_{32}',\ell_{12}',\ell_{23}'),\ c_3'=\max(d_{13}',d_{23}',\ell_{13}',\ell_{23}').
}
Assuming $u_i,w_i\in\R_{\geqslant 0}$ for $i\in\{1,2,3\}$, we have that the $p$-factor of the Euler product is 
\est{
1+O\Big(\sum_{1\leqslant i\leqslant 3}p^{-(u_i+w_4)}+\sum_{1\leqslant i,j\leqslant  3,\atop i\neq j}p^{-(u_i+w_j)}+\sum_{1\leqslant i< j\leqslant 3}p^{-(w_i+w_j)}\Big)
}
and thus both sides of~\eqref{e_p} converge whenever each of the exponents above are smaller than $-1$. 

As for~\eqref{bound1}, we notice that by symmetry it suffices to consider the contribution to the series coming from the terms with $b_1c_1\geqslant b_2c_2, b_3c_3$. This is less or equal than the left hand side of~\eqref{e_p} with $(u_1,w_1)$ replaced by $(u_1-\kappa,w_1-\kappa)$ and the lemma follows.
\end{proof}
\begin{lemma}\label{euler}
With the same notations as in Lemma~\ref{mobius}, let
\es{\label{eulerCa}
\mathfrak S^*_{\ba}:=\sum_{ q\geqslant 1}\frac{\varphi(q)}{q^{3}} \sum_{\be,\bl\in\N^3 \atop \bd\in\N^6}\frac{(a_1b_1c_1,q)(a_2b_2c_2,q)(a_3b_3c_3,q)}{b_1b_2b_3c_1c_2c_3c_4}.
}
Then $\mathfrak S^*_{(1,1,1)}=\mathfrak S_2$ with $\mathfrak S_2$ as in Theorem~\ref{2}.
\end{lemma}
\begin{proof}
With the same notations as in~\eqref{euler_exp}, we have that for $\ba=\bone$, the right hand side of~\eqref{eulerCa} is equal to
$$
\prod_p\Bigg(\sum_{\substack{q'\in\Z_{\geqslant 0},\,\bd'\in\{0,1\}^6\\ \be',\bl'\in\{0,1\}^3}}{\frac{(p-1)^{\rho_{q'}}(-1)^{e_1'+e_2'+e_3'+d_{12}'+d_{13}'+d_{21}'+d_{23}'+d_{31}'+d_{32}'+\ell_{12}'+\ell_{13}'+\ell_{23}'}}{p^{2q'+\rho_{q'}+b_1'+b_2'+b_3'+c_1'{}+c_2'{}+c_3'{}+c_4'{}-\min(q',b_1'+c_1')-\min(q',b_2'+c_2')-\min(q',b_3'+c_3')}}}\Bigg)
$$
where $\rho_{0}=0$ and $\rho_{q'}=1$ if $q'\geqslant 1$. As in lemma~2.7 of~\cite{BBS}
we observe that the terms with $q'\geqslant 2$ do not contribute. Indeed, if $q'\geqslant 2$, then $\min(q',b_i'+c_i')=b_i'+c_i'$ so that the exponent of $p$ above is $2q'+c_4'$ and so it does not depend on $d_{12}'$. In particular the contributions of $d_{12}'=0$ and $d_{12}'=1$ cancel out. After restricting the sum over $q'$ to $q'\in\{0,1\}$, we are just left with performing a finite computation which gives the claimed Euler product formula for $\mathfrak S^{*}_{(1,1,1)}$.
\end{proof}

We are now ready to prove our Theorem~\ref{2}.
\begin{proof}[Proof of Theorem~\ref{2}]
Let $K^{*}_\ba(B)$ as in~\eqref{def_K} but imposing also the coprimality conditions~\eqref{cop_cond}. In particular, by~\eqref{expwh} we have 
\es{\label{fs1}
&N_{\widehat{W}_3,\widehat{H}}(B)=K_{(1,1,1)}^*(B)+N'_{\widehat{W}_3,\widehat{H}}(B),
}
where $N'_{\widehat{W}_3,\widehat{H}}(B)$ counts the number of points in $\widehat{W}_3$ of height less than $B$  which also satisfy $x_1x_2x_3=0$. 
Now, for $x_3=0$ (and thus $y_3=1$) then~\eqref{eq} reduces to $\frac{x_1}{y_1}+\frac{x_2}{y_2}=0$. Since we have $(x_1,y_1)=(x_2,y_2)=1$, then $x_1=-x_2$, $y_1=y_2$ and thus
\es{\label{fs2}
N'_{\widehat{W}_3,\widehat{H}}(B)=1+3\sum_{x\in\Z_{\neq0},y\in\N,\,(x,y)=1\atop |x|,y\leqslant B^{1/2}}1=\frac{36}{\pi^2}B+O\left(B^\frac12\log B\right).
}
In particular, it suffices to prove an asymptotic formula with power saving error term for $K_{(1,1,1)}^*(B)$. Since it doesn't introduce any difficulties, in the following we shall compute an asymptotic formula for~$K_{\ba}^*(B)$ for all $\ba \in \mathbb{Z}^3_{\neq 0}$.

Let $0<\delta<\frac12$. Using the same approach used for Proposition~\ref{mlc} and with the notations of Lemma~\ref{mobius} and of the proof of Proposition~\ref{mlc}, we need to compute an asymptotic formula for 
\est{
&K_{\ba}^{*\pm}(B)=\sumprime_{(\bx,\bz)\in \Z_{\neq0}^3\times\Z^4_{>0}\atop a_1x_1z_1+a_2x_2z_2+a_3x_3z_3=0} \prod_{I\seq S_3}F^{\pm}_{\delta}\left(\frac{(z_1z_2z_3z_4)^{|J|}}{B}\prod_{i\in I}|x_i|\prod_{j\in J}z_j^{-1}\right).
}
With the notations of Lemma~\ref{mobius}, we can rewrite this as
\est{
&K_{\ba}^{*\pm}(B)=\sum_{\be,\bl\in\N^3,\atop \bd\in\N^6}\mu(\be,\bd,\bl)\hspace{-1em}\sum_{(\bx,\bz)\in \Z_{\neq0}^3\times\Z^4_{>0}\atop a_1^*x_1z_1+a_2^*x_2z_2+a_3^*x_3z_3=0} \prod_{I\seq S_3}F^{\pm}_{\delta}\left(\frac{(c_1z_1c_2z_2c_3z_3c_4z_4)^{|J|}}{B}\prod_{i\in I}|b_ix_i|\prod_{j\in J}(c_jz_j)^{-1}\right),
}
where $\ba^*:=(a_1^*,a_2^*,a_3^*)=(a_1b_1c_1,a_2b_2c_2,a_3b_3c_3)$ with $b_i,c_i$ as in~\eqref{bc}. Thus, proceeding as in Proposition~\ref{mlc} and using the same notations and lines of integrations we find the following expression for $K_{\ba}^{*\pm}(B)$
\est{
&\sum_{\be,\bl\in\N^3,\atop \bd\in\N^6} \frac{\mu(\be,\bd,\bl)}{(2\pi i)^8}\iint_{(c_{I})}\frac{B^{\frac12(\alpha_1+\alpha_2+\alpha_3-\xi)}\zeta(\tfrac12(\alpha_1+\alpha_2+\alpha_3+3\xi))\mathcal A_{\ba^*}(\balpha-\xi,\balpha+\xi)}{c_4^{\frac12(\alpha_1+\alpha_2+\alpha_3+3\xi)}b_1^{\alpha_1-\xi}b_2^{\alpha_2-\xi}b_3^{\alpha_3-\xi}c_1^{\alpha_1+\xi}c_2^{\alpha_2+\xi}c_3^{\alpha_3+\xi}}
\prod_I\,\tilde F_\delta^\pm (s_I)\mbox{d}s_I.
}
 Notice that by~\eqref{bound1} with $\kappa=0$ the outer series converges absolutely. We keep following the same approach as the proof of Proposition~\ref{mlc} splitting $\mathcal A_{\ba}$ into $\mathcal {M}_{\ba}+\mathcal {E}_{\ba}$ and thus $K^{*\pm}_\ba(B)$ into $M_{\ba}^{*\pm}(B)+E_{\ba}^{*\pm}(B)$. We can treat $E_{\ba}^{*\pm}(B)$ as above with the only difference that in this case we move $c_{S_3}$ to $\frac14-2\gamma+6\eps$, where $\gamma:=\frac{391 - \sqrt{152737}}{108}$ (this value is the smallest one can take under the condition that the inequalities~\eqref{cond_b_dirichlet} are satisfied). With this choice,~\eqref{ml_bound} and~\eqref{bound1} give the bound
\est{
E_{\ba}^{*\pm}(B)&\ll\! \sum_{\be,\bl\in\N^3,\atop \bd\in\N^6}\frac{\max(a_1b_1c_1,a_2b_2c_2,a_3b_3c_3)^{\frac{378\gamma}{2-27\gamma}}}{{c_4^{1+12\eps}(b_1b_2b_3)^{\frac{2}{3}-2\gamma+9\eps}c_1^{\frac23}c_2^{\frac23}c_3^{\frac23}}} 
B^{1-2\gamma+13\eps}\!\iint_{(c_{I})}\!\!\big(1+\max_{I}|s_I|\big)^{{\frac{567\gamma}{2-27\gamma}}}
\prod_I\, \big|\tilde F_\delta^\pm (s_I)\big| \mbox{d}s_I\\
&\ll_A B^{1-2\gamma+13\eps}\delta^{-{\frac{567\gamma}{2-27\gamma}}-\eps},
}
where $A:=\max_{1\leqslant i\leqslant 3} |a_i|$.
As for $M_{\ba}^{*\pm}(B)$, we treat it exactly as in Proposition~\ref{mlc}, splitting it into $M_{\ba}^{*\pm}(B)=\sum_{k=0}^3E_{\ba,k}^{*\pm}(B)$. As above we have that $M_{\ba,0}^{*\pm}(B),M_{\ba,1}^{*\pm}(B)\ll_{A} B^{\frac{25}{27}+8\eps}\delta^{-1-\eps}$
 where the sums in the error terms are immediately seen to be convergent by~\eqref{boundQ} and~\eqref{bound1}. For $M_{\ba,3}^{*\pm}(B)$ we find similarly as in the proof of Proposition~\ref{mlc}
\est{
M_{\ba,3}^{*\pm}(B)&=\sum_{\be,\bl\in\N^3,\atop \bd\in\N^6}\frac{\mu(\be,\bd,\bl)}{(2\pi i)^6}\iint_{(c_{I})}\Res_{\xi=0}\bigg(\frac{B^{1+\xi}\zeta(1+3\xi)\zeta(1+2\xi)^3\mathcal {Q}_{\ba^*,3}({\balpha,\xi})}{3s_{123}\,s\,c_4^{1+3\xi}b_1^{\alpha_1-\xi}b_2^{\alpha_2-\xi}b_3^{\alpha_3-\xi}c_1^{\alpha_1+\xi}c_2^{\alpha_2+\xi}c_3^{\alpha_3+\xi}} \bigg)\,
\prod_{I\neq \emptyset, S_3}\frac{\mbox{d}s_I}{s_I}\\
&\quad+O_{A}\left(B^{\frac{25}{27}+8\eps}\delta^{-1-\eps}+B^{1+5\eps}\delta^{1-7\eps}\right)
}
with $s:=\frac13-\frac13\sum_{I\neq \emptyset,S_3}(3-|I|)s_I$, $s_{123}:=\frac23-\frac13\sum_{I\neq S_3}|I|s_I$, lines of integration $c_{I}=\frac1{12}$ for all $I\neq S_3,\emptyset$, $\mathcal{Q}_{\ba^*,3}$ as in~\eqref{defQ} and $\boldsymbol{\alpha}$ and $\xi$ given by (\ref{nn}) with $s_{123}$ and $s$ replaced by (\ref{ss}). 
Computing the residue then gives
\est{
M_{\ba,3}^\pm(B)&=BP^*_{\ba}(\log B)+O_{A}\left(B^{\frac{25}{27}+8\eps}\delta^{-1-\eps}+B^{1+5\eps}\delta^{1-7\eps}\right)
}
where $P^*_{\ba}$ is a degree $3$ polynomial with leading constant
\est{
P^*_{\ba,3}&:=\sum_{\be,\bl\in\N^3,\atop \bd\in\N^6}\frac{\mu(\be,\bd,\bl)}{(2\pi i)^6}\iint_{(\frac1{12})}\frac{\mathcal {Q}_{\ba^*,3}({\balpha,0})}{432\,s_{123}\,s\,c_4b_1^{\alpha_1}b_2^{\alpha_2}b_3^{\alpha_3}c_1^{\alpha_1}c_2^{\alpha_2}c_3^{\alpha_3}} \,
\prod_{I\neq \emptyset, S_3}\frac{\mbox{d}s_I}{s_I}
}
and $\alpha_2,\alpha_3$ as in~\eqref{af} and $\alpha_1=2-\alpha_2-\alpha_3$. Now, 
\est{
\mathcal {Q}_{\ba^*,3}({\balpha,0})&=2\pi^\frac12\sum_{\ell=1}^\infty\frac{(a_1b_1c_1,\ell)(a_2b_2c_2,\ell)(a_3b_3c_3,\ell)\varphi(\ell)}{\ell^{3}}\prod_{i=1}^3\frac{\Gamma(\frac{1-\alpha_i}2)}{|a_i|^{1-\alpha_i}\Gamma(\frac{\alpha_i}2)}.
}
and hence by the same computation as in Proposition~\ref{mlc} we find $P^*_{\ba,3}=\frac1{144}\mathfrak S^*_\ba \mathcal  I_\ba$ with $\mathfrak S^*_\ba$ as in Lemma~\ref{euler}.

Finally, we treat $M_{\ba,2}^{*\pm}(B)$ exactly as in Proposition~\ref{mlc} and so, collecting the various asymptotics and bounds, we arrive to
\est{
K_{\ba}^{*\pm}(B)&=BP^*_{\ba}(\log B)+W^*_\ba B+O_A\left(B^{1+5\eps}\xi^{1-8\eps}+B^{1-2\gamma+13\eps}\delta^{-{\frac{567\gamma}{2-27\gamma}}-7\eps}\right)
}
for a certain $W^*_\ba \in \R$. Choosing $\delta=B^{{-\frac{\gamma (2 -27 \gamma)}{1 + 270 \gamma}}}$  we obtain 
\est{
 K^*_{\ba}(B):=BP_{\ba}^*(\log B)^3+W^*_\ba B+O_A\left(B^{1-{\frac{\gamma (2 -27 \gamma)}{1 + 270 \gamma}}+13\eps}\right)
}
In particular, by~\eqref{fs1},~\eqref{fs2} and Lemma~\ref{euler} and recalling that $\gamma={\frac{391 - \sqrt{152737}}{108}}$ we obtain Theorem~\ref{2} for all $\xi_2\leqslant{\frac{1165 - 3 \sqrt{152737}}{3264} = 0.00228169\dots}$ 
\end{proof}

\section{Proof of Theorem~\ref{mtt1}}\label{pmt}
By~\eqref{ssf} and renaming for simplicity $z_6=d_1$, $z_5=d_2$, $z_3=d_3$ and $z_4$ by $z_3$  we have to count the solutions to 
\est{
x_1d_1+x_2d_2+x_3d_3=0
}
where $\bx\in\Z_{\neq0}^3, \bd,\bz\in\N^3$ with 
$\bx=(x_1,x_2,x_3),\bd=(d_1,d_2,d_3),\bz=(z_1,z_2,z_3)$ 
subject to the coprimality conditions 
\es{\label{cop2}
&(d_i,d_j)=(z_i,z_j)=(d_k,z_k)=1\quad \forall\ i,j,k\in\{1,2,3\},\ i\neq j,\\
&(x_1,x_2,x_3)=(x_i,x_j,z_k)=1\quad \forall i,j,k \text{ such that } \{i,j,k\}=\{1,2,3\},
}
and
\est{
\max_{1\leqslant i,j \leqslant 3}\left\{|x_iz_i|^2d_1d_2d_3\frac{z_j}{d_j}\right\}\leqslant B.
}
Let $0<\delta<\frac12$. This parametrization and (\ref{ssf}) then imply that we just need to consider 
\es{\label{dfnn}
N^{\pm}_\delta (B):=2\sumprime_{\bx\in\Z^3_{\neq0},\bd,\bz\in\N^3\atop x_1d_1+x_2d_2+x_3d_3=0} \prod_{1\leqslant i,j\leqslant 3}F_{\delta}^{\pm} \left((x_iz_i)^2\frac{d_1d_2d_3}{d_j}\frac{z_j}{B}\right),
}
where $\sum'$ indicates the coprimality conditions~\eqref{cop2}, 
since for all $\delta>0$ we have
\est{
N^{-}_\delta (B)\leqslant N_{\widetilde{W}_3,\widetilde{H}}(B) \leqslant N^{+}_\delta (B).
}
We shall prove 
\es{\label{casc}
N^{\pm}_\delta (B)= B P_1(\log B)+O\left(B^{1+\eps} \delta^{1-C\eps}+B^{1-K}\delta^{-C}\right)
}
where $P_1$ is a polynomial of degree $4$ with leading coefficient $\frac{\mathfrak S_1 \cdot \mathcal I}{144}$ with the notations of Theorem \ref{mtt1} and for some $C,K>0$ and $\eps>0$ small enough, so that choosing $\delta=B^{-K/(C+1)}$, we obtain Theorem \ref{mtt1}.

\subsection{Initial manipulations}\label{iman}
We write the $F_{\delta}^{\pm}$ in term of its Mellin transform using the variable $s_{ij}$ for the cut-off function corresponding to $(i,j)$ and choosing
\est{
c_{s_{1j}}=\tfrac19+\eps,\quad c_{s_{2j}}=\tfrac19+4\eps, \quad c_{s_{3j}}=\tfrac19+6\eps}
 as lines of integration  for all $j\in \{1,2,3\}$.  We obtain
\est{
N^{\pm}_\delta (B)&=2\sumprime_{\bx\in\Z^3_{\neq0},\bd,\bz\in\N^3\atop x_1d_1+x_2d_2+x_3d_3=0}\frac1{(2\pi i)^9} \iint_{(c_{s_{ij}})} \frac{B^{\sum_{i,j}s_{ij}}}{\prod_{k}x_k^{2\sum_js_{kj}}d_k^{\sum_{i,j\, j\neq k}s_{ij}}z_k^{2\sum_{j}s_{kj}+\sum_{i}s_{i,k}}}\prod_{i,j}\tilde F_\delta^\pm (s_{ij})\mbox{d}s_{ij}\\
&=2\sumprime_{\bx\in\Z^3_{\neq0},\bd,\bz\in\N^3\atop x_1d_1+x_2d_2+x_3d_3=0} \frac1{(2\pi i)^9}\iint_{(c_{s_{ij}})} \frac{B^{s^*}}{\prod_{k}x_k^{\alpha_k-\xi_k}d_k^{\alpha_k+\xi_k}z_k^{s^*-2\xi_k}}\prod_{i,j}\tilde F_\delta^\pm (s_{ij})\mbox{d}s_{ij}
}
where
\est{
&\alpha_{k}:=\frac12\sum_{1 \leqslant i,j\leqslant 3}s_{ij}+
\sum_{1\leqslant j \leqslant 3}s_{kj}-
\frac12\sum_{1 \leqslant i \leqslant 3}s_{i,k}
,\qquad \xi_k:=\frac12\sum_{1 \leqslant i,j \leqslant 3 \atop j\neq k}s_{ij}-\sum_{1 \leqslant j \leqslant 3} s_{kj},\qquad \forall k\in\{1,2,3\}\\
&s^*:=\sum_{1 \leqslant i,j \leqslant 3}s_{ij}
}
so that 
\est{
\alpha_k-\xi_k=2\sum_{1 \leqslant j \leqslant 3}s_{kj},\qquad \alpha_k+\xi_k=\sum_{1 \leqslant i,j \leqslant 3 \atop  j\neq k}s_{ij}.
}
Note that we have
\begin{equation}
\xi_1+\xi_2+\xi_3=0, \quad \alpha_1+\alpha_2+\alpha_3=2s^*
\label{da}
\end{equation}
and that, like in the proof of Theorem \ref{2}, we are neglecting the dependencies of these notations on variables $s_{ij}$ in order to simplify the exposition.
Also, notice that with the above lines of integration we have
\es{\label{lit}
&\Re(\xi_1)
=8\eps,\quad \Re(\xi_2)
=-\eps,\quad \Re(\xi_3)
=-7\eps,\qquad \Re(s^*)=1+33\eps
\\
&\Re(\alpha_1)=\frac23+14\eps,
\quad \Re(\alpha_2)=\frac23+23\eps,
\quad \Re(\alpha_3)=\frac23+29\eps
}
so that in particular the above series are absolutely convergent by Lemma~\ref{ml}.\\
\indent
We make a change of variables, discarding the variables $s_{11},s_{22},s_{13},s_{23},s_{33}$, and introducing the variables 
$\alpha_1,\alpha_2,\xi_1,\xi_2$ and $s^*$. The inverse transformations are
\as{
&s_{11}= s^*-\alpha_1 - \xi_1 - s_{21} - s_{31} ,\qquad
&&s_{22} = s^* -\alpha_2 - \xi_2 - s_{12} - s_{32} ,\notag\\
&s_{13} =  -  s^*+\tfrac 32 \alpha_1 +\tfrac12 \xi_1 -  s_{12} +  s_{21} +  s_{31},\quad
&&s_{23} =  -  s^*+ \tfrac32 \alpha_2 +\tfrac12\xi_2 + s_{12} -  s_{21} +  s_{32},\label{invt}\\
&s_{33}= s^* -\tfrac12(\alpha_1 + \alpha_2- \xi_1 -\xi_2) -  s_{31} -  s_{32 }\notag
}
and the Jacobian is equal to $1$. In the following, to simplify the exposition, we shall keep using also the older variables (as well as $\xi_3$ and $\alpha_3$ given by (\ref{da})), treating them as function of the new ones. 

\subsection{Resolving the coprimality conditions} For $\Re(\alpha_k\pm\xi_k)>\frac23$ and $\Re(s^*-2\xi_k)>1$, using M\"obius inversion formula to remove the coprimality conditions~\eqref{cop2} (this is lemma~2.1 of~\cite{BBS}) we obtain
\es{\label{arm}
&\sumprime_{\bx\in\Z^3_{\neq0},\bd,\bz\in\N^3\atop x_1d_1+x_2d_2+x_3d_3=0}\prod_{k=1}^3\frac1{x_k^{\alpha_k-\xi_k}d_k^{\alpha_k+\xi_k}z_k^{s^*-2\xi_k}}\\[-0.5em]
&\hspace{5em}=
\sum_{\bb,\bc,\bff,\bg\in\N^3,\atop h\in\N}\sum_{\bx\in\Z^3_{\neq0},\bd,\bz\in\N^3\atop \sum_{k=1}^3 r_{k}x_kd_k=0}\mu(h)\prod_{k=1}^3\frac{\mu(b_k)\mu(c_k)\mu(f_k)\mu(g_k)}{( r_{1,k}x_k)^{\alpha_k-\xi_k}( r_{2,k}d_k)^{\alpha_k+\xi_k}( r_{3,k}z_k)^{s^*-2\xi_k}}\\
&\hspace{5em}=
\sum_{\bb,\bc,\bff,\bg\in\N^3,\atop h\in\N}\mu(h)\Bigg(\prod_{k=1}^3\frac{\mu(b_k)\mu(c_k)\mu(f_k)\mu(g_k)}{ r_{1,k}^{\alpha_k-\xi_k} r_{2,k}^{\alpha_k+\xi_k} r_{3,k}^{s^*-2\xi_k}}\zeta(s^*-2\xi_k)\Bigg)\mathcal A_{\br}(\balpha-\boldsymbol{\xi},\balpha+\boldsymbol{\xi}),\\
}
with the notation of Lemma \ref{ml} and where for $\{i,j,k\}=\{1,2,3\}$ we defined 
\begin{equation}
 r_{1,k}:=[g_i,g_j,h],\qquad
 r_{2,k}:=[b_i,b_j,f_k],\qquad
 r_{3,k}:=[c_i,c_j,f_k,g_k],\qquad  r_k:= r_{1,k} r_{2,k}.
\label{rrr}
\end{equation}
\indent
For future use we also observe that for $\sigma\geqslant\frac12+\eps$, with $\eps>0$, we have
\as{
\sum_{\bb,\bc,\bff,\bg\in\N^3,\atop h\in\N}\prod_{k=1}^3\frac{1}{(r_{1,k}r_{2,k}r_{3,k})^\sigma}
&=\prod_{p}\Bigg(
\sum_{\bb',\bc',\bff',\bg'\in\N^3,\atop h'\in\N}p^{-\sigma\sum_k\big(\max(g_i',g_j',h')+\max(b_i',b_j',f_k')+\max(c_i',c_j',f_k',g_k')\big)}\Bigg)\notag\\
&=\prod_{p}
(1+O(p^{-2\sigma}))\ll 1,\label{sbd}
}
where, in the sum over $k$ in the first line, $i,j$ are such that $\{i,j,k\}=\{1,2,3\}$.

Now, by~\eqref{arm}, we have 
\est{
N^{\pm}_\delta (B)&=2\sum_{\bb,\bc,\bff,\bg\in\N^3,\atop h\in\N}\frac{\mu(h)}{(2\pi i)^9} \iint_{(\cdotsp)} B^{s^*}\Bigg(\prod_{k=1}^3\frac{\mu(b_k)\mu(c_k)\mu(f_k)\mu(g_k)}{ r_{1,k}^{\alpha_k-\xi_k} r_{2,k}^{\alpha_k+\xi_k} r_{3,k}^{s^*-2\xi_k}}\zeta(s^*-2\xi_k)\Bigg)\\
&\hspace{11em}\times\mathcal A_{\br}(\balpha-\boldsymbol{\xi},\balpha+\boldsymbol{\xi})\bigg(\prod_{i,j}\tilde F_\delta^\pm (s_{ij})\bigg)\mbox{d}(\cdotsp),
}
{where, here and below, we indicate by $\int_{(\cdotsp)} \mbox{d}(\cdotsp)$ an integral with respect to the variables $s^*,$$\alpha_1$,$\alpha_2$,$\delta_1$,$\delta_2$ and $s_{12}$,$s_{21}$,$s_{31}$,$s_{32}$, along the lines of integration previously indicated, with the exclusion of the variables which have been eliminated by the computation of a residue.}

\subsection{Applying Lemma~\ref{ml}} 
\label{73}

We write $\mathcal A_{\br}(\balpha-\boldsymbol{\xi},\balpha+\boldsymbol{\xi})$ as $\mathcal M_{\br}(\balpha-\boldsymbol{\xi},\balpha+\boldsymbol{\xi})+\mathcal E_{\br}(\balpha-\boldsymbol{\xi},\balpha+\boldsymbol{\xi})$ and we split accordingly $N^{\pm}_\delta (B)$ into 
\es{\label{fsp}
N^{\pm}_\delta (B)=M^{\pm}_\delta (B)+E^{\pm}_\delta (B). 
}
Differently from the case of Theorem~\ref{2}, here $E^{\pm}_\delta (B)$ also contributes to a main term, of size $B$, which can be extracted as follow.

We move the lines of integration $c_{\xi_1}$, $c_{\xi_2}$ and  $c_{s^*}$  in the integrals defining $E^{\pm}_\delta (B)$ to $c_{\xi_1}=2K$, $c_{\xi_2}=-K$ and $c_{s^*}=1-K$ for some fixed real number $K>0$ small enough, passing through the simple pole of the integrand at 
$s^*=1+2\xi_1$. If $K$ is sufficiently small then we don't pass through any other pole and we stay inside the region~\eqref{hre} where  $\mathcal E_{\br}$ is holomorphic and where the sums are absolutely convergent. For the integral on the new lines of integration we use~\eqref{ml_bound} and a trivial bound for $\zeta$ and we obtain that, for $K$ small enough, the integral is bounded by
\as{
&\ll B^{1-K}\sum_{\bb,\bc,\bff,\bg\in\N^3,\atop h\in\N} \iint_{(\cdotsp)} \Bigg(\prod_{k=1}^3\frac{(r_k(|s^*|+|\xi_k|)(|\alpha_k|+|\xi_k|))^{C_1K}}{ (r_{1,k} r_{2,k}r_{3,k})^{\frac23-K}}\Bigg)
\bigg|\prod_{i,j}\tilde F_\delta^\pm (s_{ij})\bigg|\mbox{d}(\cdotsp)\notag\\
&\ll B^{1-K}\delta^{-C_2K},\notag
}
where the second line is obtained as for~\eqref{dax} using~\eqref{bounds_cut-off} and~\eqref{sbd}, after reintroducing the original variables $s_{ij}$.
Also, {we remind that,} here and below, $C_1,C_2,C_3,\dots$ will indicate fixed positive real numbers.

Collecting the contribution of the residue we obtain
\est{
E^{\pm}_\delta (B)&=2\!\!\!\sum_{\bb,\bc,\bff,\bg\in\N^3,\atop h\in\N}\!\!\!\frac{\mu(h)}{(2\pi i)^8} \iint_{(\cdotsp)} B^{1+2\xi_1}\Bigg(\prod_{k=1}^3\frac{\mu(b_k)\mu(c_k)\mu(f_k)\mu(g_k)}{ r_{1,k}^{\alpha_k-\xi_k} r_{2,k}^{\alpha_k+\xi_k} r_{3,k}^{1+2\xi_1-2\xi_k}}\Bigg)\zeta({1+}2\xi_1-2\xi_2)\zeta({1+}4\xi_1+2\xi_2)\\
&\hspace{11em}\times\mathcal E_{\br}(\balpha-\boldsymbol{\xi},\balpha+\boldsymbol{\xi})\bigg(\prod_{i,j}\tilde F_\delta^\pm (s_{ij})\bigg)\mbox{d}(\cdotsp)+O\left(B^{1-K}\delta^{-C_2K}\right),
}
since $\xi_3=-\xi_1-\xi_2$ and with the notations (\ref{invt}). We then move the line of integration $c_{\xi_1}$ to $c_{\xi_1}=-K/2$ passing through the pole at $\xi_1=-\xi_2/2$.  The integral on the new lines of integration can be bounded as above, whereas in the integral coming from the contribution of the residue, we move $c_{\xi_2}$ to $c_{\xi_2}=K$ passing through a simple pole at $\xi_2=0$ (in which case $\xi_3=0$). Bounding once again the contribution of the integral as above we arrive to
\est{
E^{\pm}_\delta (B)&=\frac{B}{6}\sum_{\bb,\bc,\bff,\bg\in\N^3,\atop h\in\N}\frac{\mu(h)}{(2\pi i)^6} \iint_{(\cdotsp)} \Bigg(\prod_{k=1}^3\frac{\mu(b_k)\mu(c_k)\mu(f_k)\mu(g_k)}{ r_{1,k}^{\alpha_k} r_{2,k}^{\alpha_k} r_{3,k}}\Bigg)\mathcal E_{\br}(\balpha,\balpha)\bigg(\prod_{i,j}\tilde F_\delta^\pm (s_{ij})\bigg)\mbox{d}(\cdotsp)\\
&\quad+O\left(B^{1-K}\delta^{-C_3K}\right).
}
The product $\prod_{i,j}\tilde F_\delta^\pm (s_{ij})$ can now be replaced by $\prod_{i,j}\frac1{s_{ij}}$ at a cost of an error which is $O\left(B \delta^{1-C_5\eps}\right)$. Indeed, by~\eqref{ml_bound} and~\eqref{sbd} we have 
\est{
&\sum_{\bb,\bc,\bff,\bg\in\N^3,\atop h\in\N}\frac{\mu(h)}{(2\pi i)^6} \iint_{(\cdotsp)} \Bigg(\prod_{k=1}^3\frac{\mu(b_k)\mu(c_k)\mu(f_k)\mu(g_k)}{ r_{1,k}^{\alpha_k} r_{2,k}^{\alpha_k} r_{3,k}}\Bigg)\mathcal E_{\br}(\balpha,\balpha)\bigg(\prod_{i,j}\tilde F_\delta^\pm (s_{ij})-\frac1{s_{ij}}\bigg)\mbox{d}(\cdotsp)\\
&\hspace{5em}\ll \iint_{(\cdotsp)} \bigg(\prod_{i,j}\bigg|\tilde F_\delta^\pm (s_{ij})-\frac1{s_{ij}}\bigg||s_{ij}|^{C_4\eps}\bigg)\mbox{d}(\cdotsp)\ll \delta^{1-C_5\eps},
}
by proceeding as in~\eqref{decb} after reintroducing {six of the variables $s_{ij}$ (with the remaining three variables kept as functions of those)}, since we now have the extra relation $\sum_{1 \leqslant i,j \leqslant 3}s_{ij}=s^*=1$.
Collecting the above computations, we then get
\es{\label{emt}
E^{\pm}_\delta (B)=B P_0+O\left(B \delta^{1-C_5\eps}+B^{1-K}\delta^{-C_6K}\right)
} for some $P_0\in\R$. 

We now move to the analysis of $M^{\pm}_\delta (B)$. Following the definition of $\mathcal M_{\br}$, we split $M^{\pm}_\delta (B)$ in the following way
\es{\label{emmt}
M_{\delta}^\pm(B)=\sum_{\bepsilon\in\{\pm1\}^3}M_{\delta,\bepsilon}^\pm(B),
}
where the sum is over $\bepsilon=(\epsilon_1,\epsilon_2,\epsilon_3)\in\{\pm 1\}^3$ and
\es{\label{inine}
M_{\delta,\bepsilon}^\pm(B)&:=\frac2{(2\pi i)^9}\iint_{(\cdotsp)}B^{s^*}\frac{\prod_{k}\zeta(s^*-2\xi_k)\zeta(1+2\epsilon_k\xi_k)}{2s^*-\epsilon_1\xi_1-\epsilon_2\xi_2-\epsilon_3\xi_3-2} \mathcal {Q_{\bepsilon}}({\balpha,\boldsymbol{\xi}})
\bigg(\prod_{i,j}\,\tilde F_\delta^\pm (s_{ij})\bigg)\,\mbox{d}(\cdotsp)
}
with lines of integration as given in Section~\ref{iman} (in particular~\eqref{lit} is satisfied) and where
\est{
\mathcal {Q_{\bepsilon}}({\balpha,\boldsymbol{\xi}})&:=
\sum_{\bb,\bc,\bff,\bg\in\N^3,\atop h,\ell\in\N}\mu(h) \Bigg(\prod_{k=1}^3\frac{\mu(b_k)\mu(c_k)\mu(f_k)\mu(g_k)}{ r_{1,k}^{\alpha_k-\xi_k} r_{2,k}^{\alpha_k+\xi_k} r_{3,k}^{s^*-2\xi_k}}\Bigg)
\frac{(r_1,\ell)^{1+2\epsilon_1\xi_1}(r_2,\ell)^{1+2\epsilon_2\xi_2}(r_3,\ell)^{1+2\epsilon_3\xi_2}}{\ell^{3+2\epsilon_1\xi_1+2\epsilon_2\xi_2+2\epsilon_3\xi_3}}\varphi(\ell)\\
&\hspace{1em}\times
2\pi^\frac12\prod_{i=1}^3\frac{\Gamma(\frac{-\alpha_i+\epsilon_i\xi_i}2+\frac{1+2s^*-\epsilon_1\xi_1-\epsilon_2\xi_2-\epsilon_3\xi_3}{6})}{r_i^{-\alpha_i+\epsilon_i\xi_i+\frac{1+2s^*-\epsilon_1\xi_1-\epsilon_2\xi_2-\epsilon_3\xi_3}{3}}\Gamma(\frac{1+\alpha_i-\epsilon_i\xi_i}2-\frac{1+2s^*-\epsilon_1\xi_1-\epsilon_2\xi_2-\epsilon_3\xi_3}{6})}
}
where we recall that $2s^*=\alpha_1+\alpha_2+\alpha_3$ and the notations (\ref{invt}). Notice that by~\eqref{sbd} and~\eqref{sfg} we have that $\mathcal {Q_{\bepsilon}}({\balpha,\boldsymbol{\xi}})$ is holomorphic and bounded for
\est{
|\Re(s^*-1)|, |\Re(\xi_i)|,|\Re(\alpha_i-\tfrac23)|<20K,
}
with $K$ small enough.

Now, we move the lines of integration $c_{s^*},c_{\xi_1},c_{\xi_2}$ in~\eqref{inine} to
$c_{s^*}=1-K$, $c_{\xi_1}=16K$, $c_{\xi_2}=-14 K$ (so that on the new lines of integration $\Re(\xi_3)=-2K$ since $\xi_1+\xi_2+\xi_3=0$) passing through simple poles at $s^*=1+2\xi_1$ and, if $\bepsilon\neq(-1,1,1)$, $(-1,-1,1)$, $(-1,1,-1)$, at $s^*=1+\frac{1}{2}\left(\epsilon_1\xi_1+\epsilon_2\xi_2+\epsilon_3\xi_3\right)$.
Indeed, the denominator has positive real part on the original lines of integrations whereas on the new lines of integrations it has real part equal to $2(-K-8\epsilon_1K+7\epsilon_2K+\epsilon_3K)$ which is negative if and only if $\epsilon_1=1$ or $\epsilon_1=\epsilon_2=\epsilon_3=-1$. Also note that we stay on the same side of the poles of the other $\zeta$ factors. Alluding to~\eqref{bounds_cut-off}, a trivial bound for $\zeta$ and the fact that $\mathcal {Q}_{\bepsilon}$ is bounded, we can use the same argument used several times in Sections~\ref{nocop}-\ref{arith} in order to bound trivially the contribution of the integral on the new lines of integration obtaining that its contribution is $O\left(B^{1-K/2}\delta^{-C_7}\right)$.

It follows that 
\est{
M_{\bepsilon}^\pm(B)= M_{\bepsilon,1}^\pm(B)+ M_{\bepsilon,2}^\pm(B),
}
where $M_{\bepsilon,1}^\pm$ denotes the contribution of the pole at $s^*=1+2\xi_1$ and $M_{\bepsilon,2}^\pm(B)$ is  the contribution of the pole at $s^*=1+\frac{1}{2}\left(\epsilon_1\xi_1+\epsilon_2\xi_2+\epsilon_3\xi_3\right)$ if $\bepsilon\notin\{(-1,1,1),(-1,-1,1),(-1,1,-1)\}$ and $M_{\bepsilon,2}^\pm(B):=0$ otherwise.

\subsection{The pole at $s^*=1+2\xi_1$ when $\bepsilon\neq(1,-1,1)$} \label{wq}
Using the fact that $\xi_3=-\xi_1-\xi_2$ we have
\as{
M_{\bepsilon,1}^\pm(B)&=\frac2{(2\pi i)^8}\iint_{(\cdotsp)}B^{1+2\xi_1}\frac{\prod_{k\neq 1}\zeta(1+2\xi_1-2\xi_k)\prod_{k=1}^3\zeta(1+2\epsilon_k\xi_k)}{(4-\epsilon_1)\xi_1-\epsilon_2\xi_2-\epsilon_3\xi_3} \notag \\
&\hspace{6em}\times\mathcal {Q_{\bepsilon}}({\balpha,\boldsymbol{\xi}})
\bigg(\prod_{i,j}\,\tilde F_\delta^\pm (s_{ij})\bigg)\,\mbox{d}(\cdotsp) \notag\\
&=\frac2{(2\pi i)^8}\iint_{(\cdotsp)}B^{1+2\xi_1}\frac{\zeta(1+2\epsilon_1\xi_1)\zeta(1+2\xi_1-2\xi_{2})\zeta(1+2\epsilon_2\xi_{2})}{(4-\epsilon_1+\epsilon_3)\xi_1+(-\epsilon_2+\epsilon_3)\xi_{2}}\label{inte} \\
&\hspace{6em}\times\zeta(1+4\xi_1+2\xi_{2})\zeta(1-2\epsilon_3(\xi_1+\xi_{2}))\mathcal {Q_{\bepsilon}}({\balpha,\boldsymbol{\xi}})
\bigg(\prod_{i,j}\,\tilde F_\delta^\pm (s_{ij})\bigg)\,\mbox{d}(\cdotsp), \notag
}
with $c_{\xi_1}=8\eps$ and $c_{\xi_2}=-\eps$.
 Notice that for $\bepsilon=(1,-1,1)$ one has a double pole when $4\xi_1+2\xi_2=0$ which causes some (mostly notational) issues when moving the lines of integration as we shall do throughout this section. For this reason, we prefer to defer to the next section the treatment of this term.

Next, we move the lines of integration $c_{\xi_1}$ and $c_{\xi_{2}}$ to $c_{\xi_1}=-K$ and $c_{\xi_2}=-4K$ passing through several poles. As before, the integral on the new lines of integration is $O\left(B^{1-K}\delta^{-C_8}\right)$. The poles we encounter are the following:
\begin{enumerate}[(a)]
\item a pole at $\xi_1=0$ which is simple if $\epsilon_2\neq\epsilon_3$ and is double if $\epsilon_2=\epsilon_3$;
\item a simple pole at $\xi_1=-\frac12\xi_2$;
\item a simple pole at $\xi_1=-\xi_2$;
\item a simple pole at $\xi_1=-\frac13\xi_2$ if $\bepsilon=(-1,-1,1)$.
\end{enumerate}
We now examine the contribution of the residue of each of these poles.

(a) We write the contribution of the residue at $\xi_1=0$ as a small circuit integral
\est{
& \frac{2B}{(2\pi i)^7}\iint_{(\cdotsp)}{\frac1{2\pi i}\oint_{|\xi_1|=\eps/2}}\Bigg(\left(1+{2\xi_1\log B}\right)\frac{\zeta(1+2\epsilon_1\xi_1)\zeta(1+2\xi_1-2\xi_{2})\zeta(1+2\epsilon_2\xi_{2})}{(4-\epsilon_1+\epsilon_3)\xi_1+(-\epsilon_2+\epsilon_3)\xi_{2}} \\
&\hspace{7em}\times\zeta(1+4\xi_1+2\xi_{2})\zeta(1\mp_{3}2(\xi_1+\xi_{2}))\mathcal {Q_{\boldsymbol{\epsilon}}}({\balpha,\boldsymbol{\xi}})
\bigg(\prod_{i,j}\,\tilde F_\delta^\pm (s_{ij})\bigg)\notag\Bigg)\,\mbox{d}(\cdotsp)
}
where we can assume that the line of integration $c_{\xi_2}$ is at $c_{\xi_2}=-\eps$ .
The next step is to observe that we can replace $\prod_{i,j}\tilde F_\delta^\pm (s_{ij})$ by $\prod_{i,j}\frac1{s_{ij}}$ at a cost of an error which is 
\es{\label{erx}
O\left(B^{1+\eps} \delta^{1-C_9\eps}\right).
}
To show this we first observe that, by the convexity bound \cite[(5.1.4)]{Tit}, on the lines of integration the integrand is \est{
\ll \log B\,(1+|\xi_1|+|\xi_2|)^{C_{10}\eps} \prod_{i,j}|\tilde F_\delta^\pm (s_{ij})|
\ll \log B \prod_{i,j}|\tilde F_\delta^\pm (s_{ij})|(1+|s_{ij}|)^{C_{11}\eps}.
}
We go back to using the $s_{ij}$ as variables (excluding for example the variable $s_{11}$ because we have a variable less) and observe we have $s^*=\sum_{1\leqslant i,j\leqslant 3}s_{ij}=1+2\xi_1=1+O(\varepsilon)$ and thus $s_{11}=1-\sum_{(i,j)\neq(1,1)}s_{ij}+O(\varepsilon)$. Thus, proceeding as for~\eqref{decb} we can replace $\prod_{i,j}\tilde F_\delta^\pm (s_{ij})$ by $\prod_{i,j}\frac1{s_{ij}}$ at a cost of an error which is bounded by~\eqref{erx}. In the end, we find that the contribution of the pole at $\xi_1=0$ is
\est{
B P_{\bepsilon,1,1}(\log B)+O\left(B^{1+\eps} \delta^{1-C_9\eps}+B^{1-K}\delta^{-C_8}\right)
}
where $P_{\bepsilon,1,1}$ is a polynomial of degree $0$ or $1$ (not depending on the choice for $\delta$ and $\pm$ in $N^{\pm}_\delta (B)$) obtained by evaluating the above integral with the $\prod_{i,j}\frac1{s_{ij}}$ instead of $\prod_{i,j}\tilde F_\delta^\pm (s_{ij})$.

(b) The contribution of the pole at $\xi_1=-\frac12\xi_2$ is  
\est{
 &\frac{-1}{(2\pi i)^7}\iint_{(\cdotsp)}B^{1-\xi_2}\frac{\zeta(1-\epsilon_1\xi_2)\zeta(1-3\xi_{2})\zeta(1+2\epsilon_2\xi_{2})\zeta(1-\epsilon_3\xi_{2})}{(4-\epsilon_1+2\epsilon_2-\epsilon_3)\xi_2} \mathcal {Q_{\bepsilon}}({\balpha,\boldsymbol{\xi}})
\bigg(\prod_{i,j}\,\tilde F_\delta^\pm (s_{ij})\bigg)\,\mbox{d}(\cdotsp). \notag
}
We move the line of integration $c_{\xi_2}$ to $c_{\xi_2}=K$ passing through a pole at $\delta_2=0$. The contribution of the integral on the new lines of integration is $O\left(B^{1-K}\delta^{-C_{12}}\right)$. For the contribution of the residue, we follow the same approach as above writing it as a small circle integral and observing that since again $s^*=\sum_{1\leqslant i,j\leqslant 3}s_{ij}=1+2\xi_1=1-\xi_2=1+O(\varepsilon)$ we can replace $\prod_{i,j}\,\tilde F_\delta^\pm (s_{ij})$ by $\prod_{i,j}\frac1{s_{ij}}$ at the cost of an error which is $O\left(B^{1+\eps}\delta^{1-C_{13}\eps}\right)$. Then, computing the integral we find that the contribution to~\eqref{inte} from the pole at $\xi_1=-\frac12\xi_2$ is 
\est{
B P_{\bepsilon,1,2}(\log B)+O\left(B^{1+\eps} \delta^{1-C_{13}\eps}+B^{1-K}\delta^{-C_{12}}\right)
}
where $P_{\bepsilon,1,2}$ is a polynomial of degree $4$ of leading coefficient 
\est{
&-\frac{\epsilon_1\epsilon_2\epsilon_3}{ 2\cdot 3\cdot (4-\epsilon_1+2\epsilon_2-\epsilon_3)}\II,
}
where
\begin{equation}
\II:= \frac1{4!}\frac1{(2\pi i)^6}\iint_{(\cdotsp)}\mathcal {Q_{\bepsilon}}({\balpha,\bzero})
\bigg(\prod_{i,j}\frac1{s_{ij}}\bigg)\,\mbox{d}(\cdotsp).
\label{intt}
\end{equation}
It is noteworthy that, since $\boldsymbol{\xi}=\bzero$, $\II$ does not depend on $\bepsilon$.

(c) The contribution of the pole at $\xi_1=-\xi_2$ is 
$$
\frac{\epsilon_3}{(2\pi i)^7}\!\!\iint_{(\cdotsp)}\!\!\!\!B^{1-2\xi_2}\frac{\zeta(1-2\epsilon_1\xi_2)\zeta(1-4\xi_{2})\zeta(1+2\epsilon_2\xi_{2})\zeta(1-2\xi_{2})}{(4-\epsilon_1+\epsilon_2)\xi_2} \mathcal {Q_{\bepsilon}}\big({\balpha,(-\xi_2,\xi_2,0)}\big)
\!\bigg(\prod_{i,j}\,\tilde F_\delta^\pm (s_{ij})\bigg)\!\mbox{d}(\cdotsp)
$$
and proceeding as above we have that this is
\est{
B P_{\bepsilon,1,3}(\log B)+O\left(B^{1+\eps} \delta^{1-C_{14}\eps}+B^{1-K}\delta^{-C_{15}}\right)
}
where $P_{\bepsilon,1,3}$ is a polynomial of degree $4$ of leading coefficient 
\est{
& \frac{\epsilon_1\epsilon_2\epsilon_3}{ 2\cdot (4-\epsilon_1+\epsilon_2)}\II,
}
with the notation (\ref{intt}).

(d) The contribution of the pole at $\xi_1=-\frac13\xi_2$ with $\bepsilon=(-1,-1,1)$ is
\est{
&\frac13\frac{1}{(2\pi i)^7}\iint_{(\cdotsp)}B^{1-\frac23\xi_2}\zeta\left(1+\tfrac23\xi_2\right)\zeta\left(1-\tfrac83\xi_{2}\right)\zeta(1-2\xi_{2})\\
&\hspace{7em}\times\zeta\left(1+\tfrac23\xi_{2}\right)\zeta\left(1-\tfrac43\xi_{2}\right)\mathcal {Q_{\bepsilon}}({\balpha,\boldsymbol{\xi}})
\bigg(\prod_{i,j}\,\tilde F_\delta^\pm (s_{ij})\bigg)
\,\mbox{d}(\cdotsp) 
}
and, proceeding as above, we have that this is
\est{
B P_{\bepsilon,1,4}(\log B)+O\left(B^{1+\eps} \delta^{1-C_{16}\eps}+B^{1-K}\delta^{-C_{17}}\right)
}
where $P_{\bepsilon,1,4}$ is a polynomial of degree $4$ of leading coefficient
\est{
&\frac{1}{ 2^4\cdot 3}\II.
}
\newline
\indent
Regrouping the above four contributions, we obtain that for $\bepsilon\neq(1,-1,1)$ we have
\est{
M_{\ba,\bepsilon,1}^\pm(B)=B P_{\bepsilon,1}(\log B)+O(B^{1+\eps} \delta^{1-C_{18}\eps}+B^{1-K}\delta^{-C_{19}})
}
where $P_{\bepsilon,1}$ is a degree $4$ polynomial with leading coefficient
\est{
\II\times
\begin{cases}
\displaystyle -\frac{\epsilon_1\epsilon_2\epsilon_3}{ 2\cdot 3\cdot (4-\epsilon_1+2\epsilon_2-\epsilon_3)} +\frac{\epsilon_1\epsilon_2\epsilon_3}{ 2\cdot (4-\epsilon_1+\epsilon_2)} & \mbox{if } \bepsilon\neq (-1,-1,1),(1,-1,1)\\[3mm]
\displaystyle\frac1{2^4}
&\mbox{if } \bepsilon=(-1,-1,1),
\end{cases}
}
with the notation (\ref{intt}).

\subsection{The pole at $s^*=1+2\xi_1$ for $\bepsilon=(1,-1,1)$}\label{wqq}For brevity, in this section we write $\bepsilon^*:=(1,-1,1)$.
 We have
\est{
M_{{\bepsilon^*}\!,1}^\pm(B)&=\frac2{(2\pi i)^8}\iint_{(\cdotsp)}B^{1+2\xi_1}\frac{\zeta(1+2\xi_1)\zeta(1+2\xi_1-2\xi_{2})\zeta(1-2\xi_{2})}{4\xi_1+2\xi_{2}} \\
&\hspace{4em}\times\zeta(1+4\xi_1+2\xi_{2})\zeta(1-2(\xi_1+\xi_{2}))\mathcal {Q}_{{\bepsilon^*}}({\balpha,\boldsymbol{\xi}})
\bigg(\prod_{i,j}\,\tilde F_\delta^\pm (s_{ij})\bigg)\,\mbox{d}(\cdotsp). \notag
}
Here, we move $c_{\xi_2}$ to $c_{\xi_2}=2K$, passing through simple poles at $\xi_2=0$ and $\xi_2=\xi_1$. For the integral on the new lines of integration, we move $c_{\xi_1}$ to $c_{\xi_1}=-K/2$ passing through a pole at $\xi_1=0$. {If $K$ is sufficiently small,} the integral on these new lines of integration can then be estimated trivially by $O\left(B^{1-K} \delta^{-C_{20}}\right)$. Thus, overall we shall compute the following residues arising for the following poles
\begin{enumerate}[(a)]
\item a simple pole at $\xi_2=0$ 
\item a simple pole $\xi_2=\xi_1$;
\item a simple pole at $\delta_1=0$ (with $c_{\xi_2}>0$).
\end{enumerate}

\begin{enumerate}[(a)]
\item The contribution of the residue at $\xi_2=0$ is
\est{
&\frac1{(2\pi i)^7}\iint_{(\cdotsp)}B^{1+2\xi_1}\frac{\zeta(1+2\xi_1)^2\zeta(1+4\xi_1)\zeta(1-2\xi_{1})}{4\xi_1}\mathcal {Q}_{{\bepsilon^*}}({\balpha,\boldsymbol{\xi}})
\bigg(\prod_{i,j}\,\tilde F_\delta^\pm (s_{ij})\bigg)\,\mbox{d}(\cdotsp) \notag
}
and, as in the previous section, one sees that this is $$BP_{{\bepsilon^*}\!,1,4}(\log B)+O\left(B^{1+\eps} \delta^{1-C_{21}\eps}+B^{1-K}\delta^{-C_{22}}\right)$$ with $P_{{\bepsilon^*}\!,4,1}$ of degree $4$ with leading coefficient
\est{
&-\frac{1}{ 2^3}\II,
}
with the notation (\ref{intt}).

\item The contribution of the residue at  $\xi_2=\xi_1$ is
\est{
&\frac1{(2\pi i)^7}\iint_{(\cdotsp)}B^{1+2\xi_1}\frac{\zeta(1+2\xi_1)\zeta(1-2\xi_{1})\zeta(1+6\xi_1)\zeta(1-4\xi_1)}{6\xi_1} \mathcal {Q}_{{\bepsilon^*}}({\balpha,\boldsymbol{\xi}})
\bigg(\prod_{i,j}\,\tilde F_\delta^\pm (s_{ij})\bigg)\,\mbox{d}(\cdotsp) \notag\\
&\hspace{4em}=BP_{{\bepsilon^*}\!,1,2}(\log B)+O\left(B^{1+\eps} \delta^{1-C_{23}\eps}+B^{1-K}\delta^{-C_{24}}\right)
}
with $P_{{\epsilon^*}\!,1,2}$ of degree $4$ with leading coefficient 
\est{
&\frac{1}{2^2\cdot 3^2}\II.
}

\item The contribution of the residue at  $\xi_1=0$ is
\est{
& \frac1{(2\pi i)^7}\iint_{(\cdotsp)}B\frac{\zeta(1-2\xi_{2})^3\zeta(1+2\xi_{2})}{2\xi_{2}}\mathcal {Q}_{{\bepsilon^*}}({\balpha,\boldsymbol{\xi}})
\bigg(\prod_{i,j}\,\tilde F_\delta^\pm (s_{ij})\bigg)\,\mbox{d}(\cdotsp) \notag\\
&\hspace{4em}=P_{{\bepsilon^*}\!,1,3}B+O\left(B^{1+\eps} \delta^{1-C_{25}\eps}+B^{1-K}\delta^{-C_{26}}\right)
}
with $P_{{\bepsilon^*}\!,1,3}\in\R$.\\
\newline
\indent
Collecting the various terms, we then find that
\est{
M_{{\bepsilon^*}\!,1}^\pm(B)&=BP_{{\bepsilon^*}\!,1}(\log B)+O\left(B^{1+\eps} \delta^{1-C_{27}\eps}+B^{1-K}\delta^{-C_{28}}\right)
}
where $P_{{\bepsilon^*}\!,1}$ is a polynomial of degree $4$ with leading coefficient
\est{
-\frac{7}{2^3\cdot 3^2}\II.
}

\end{enumerate}

\subsection{The pole at $s^*=1+\frac{1}{2}\left(\epsilon_1\xi_1+\epsilon_2\xi_2+\epsilon_3\xi_3\right)$.}\label{wqqq} Recall that $M_{\bepsilon,2}^\pm(B):=0$ if $\bepsilon=(-1,1,1)$, $(-1,-1,1)$, $(-1,1,-1)$. In all other cases we have 
\est{
M_{\bepsilon,2}^\pm(B)&=\frac1{(2\pi i)^8}\iint_{(\cdotsp)}B^{1+\frac{1}{2}\left(\epsilon_1\xi_1+\epsilon_2\xi_2+\epsilon_3\xi_3\right)}\prod_{k}\zeta\left(1+\frac{1}{2}\left(\epsilon_1\xi_1+\epsilon_2\xi_2+\epsilon_3\xi_3\right)-2\xi_k\right)\zeta(1+2\epsilon_k\xi_k)\\
&\hspace{7em}\times\mathcal {Q_{\bepsilon}}({\balpha,\boldsymbol{\xi}})
\bigg(\prod_{i,j}\,\tilde F_\delta^\pm (s_{ij})\bigg)\,\mbox{d}(\cdotsp).
}
%
If $\epsilon_1=\epsilon_2=\epsilon_3$, then the exponent of $B$ is $1$. In particular, since in this case we have the relation $s^*=\sum_{i,j}s_{ij}=1$ because $\xi_1+\xi_2+\xi_3=0$, we can replace $\prod_{i,j}\tilde F_\delta^\pm (s_{ij})$ by $\prod_{i,j} s_{ij}^{-1}$ at a cost of an error which is $O\left(B^{1+\eps} \delta^{1-C_{25}\eps}\right)$ like in Section~\ref{73}. Thus, for $\bepsilon=(1,1,1),(-1,-1,-1)$ we have
\est{
M_{\bepsilon,2^\pm(B)}=P_{\bepsilon,2,1} B +O\left(B^{1+\eps} \delta^{1-C_{29}\eps}\right)
}
for some $P_{\bepsilon,2,1}\in\R$. 

Therefore, we are left with considering the cases where $\epsilon_r=-\epsilon_{k_1}=-\epsilon_{k_2}$ with $\{r,k_1,k_2\}=\{1,2,3\}$ and $k_1<k_2$. Notice in particular that since $\xi_1+\xi_2+\xi_3=0$, then we have $\epsilon_1\xi_1+\epsilon_2\xi_2+\epsilon_3\xi_3= 2\epsilon_r\xi_r$. Thus, 
\est{
M_{\bepsilon,2^\pm(B)}&=\frac1{(2\pi i)^8}\iint_{(\cdotsp)}B^{1+\epsilon_r\xi_r}\zeta(1+(\epsilon_r-2)\xi_r)\zeta(1+2\epsilon_r\xi_r)
\zeta(1+\epsilon_r\xi_r-2\xi_{k_1})\zeta(1+2\epsilon_{k_1}\xi_{k_1})\\
&\hspace{1em}\times\zeta(1+(\epsilon_r+2)\xi_r+2\xi_{k_1})\zeta(1-2\epsilon_{k_2}(\xi_r+\xi_{k_1}))\mathcal {Q_{\bepsilon}}({\balpha,\boldsymbol{\xi}})
\bigg(\prod_{i,j}\,\tilde F_\delta^\pm (s_{ij})\bigg)\,\mbox{d}(\cdotsp).
}
Notice that we made a change of variables (of jacobian $\pm 1$, since $\xi_1+\xi_2+\xi_3=0$) using $\xi_r,\xi_{k_1}$ rather than $\xi_1,\xi_2$.

\indent
Next, we move the lines of integration $c_{\xi_r}$, $c_{\xi_{k_1}}$ to $c_{\xi_r}=c_{\xi_{k_1}}=-\epsilon_r K$. In doing so we pass through a double pole at $\xi_r=0$ and, if $r=1$, through the simple poles of the $\zeta$ factors on the second line at $\xi_1=-\frac23\xi_{2}$ and $\xi_1=-\xi_{2}$. Thus, as in the previous sections, if $\bepsilon=(1,-1,1),(1,1,-1)$ we find
\est{
M_{\bepsilon,2^\pm(B)}=B P_{\bepsilon,2,2}(\log B)+O\left(B^{1+\eps} \delta^{1-C_{30}\eps}+B^{1-K}\delta^{-C_{31}}\right)
}
with $P_{\bepsilon,2,2}(\log B)$ of degree at most $1$
and the same holds for the contribution of the pole at $\xi_r=0$ when $r=1$. We are therefore left with computing the contributions of the two remaining poles when $\bepsilon=(1,-1,-1)$ and $(r,k_1,k_2)=(1,2,3)$. The contribution of the pole at $\xi_1=-\frac23\xi_{2}$ is 
\est{
&\frac13\frac1{(2\pi i)^7}\iint_{(\cdotsp)}B^{1-\frac23\xi_{2}}\zeta\left(1+\tfrac23\xi_{2}\right)^2\zeta\left(1-\tfrac43\xi_2\right)
\zeta\left(1-\tfrac83\xi_{2}\right)\zeta(1-2\xi_{2})\mathcal{Q}_{(1,-1,-1)}({\balpha,\boldsymbol{\xi}})
\bigg(\prod_{i,j}\,\tilde F_\delta^\pm (s_{ij})\bigg)\,\mbox{d}(\cdotsp),
}
and this is $B P_{(1,-1,-1),2,3}(\log B)+O\left(B^{1+\eps} \delta^{1-C_{31}\eps}+B^{1-K}\delta^{-C_{32}}\right)$ with $P_{(1,-1,-1),2,3}$ a polynomial of degree $4$ and leading coefficient
\est{
&\frac{1}{ 2^4\cdot3}\II,
}
with the notation (\ref{intt}).\\
\indent
The contribution of the pole at $\xi_1=-\xi_{2}$ is
\est{
&\frac12\frac1{(2\pi i)^7}\iint_{(\cdotsp)}B^{1-\xi_2}\zeta(1+\xi_2)\zeta(1-2\xi_2)^2
\zeta(1-3\xi_{2})\zeta(1-\xi_{2})\mathcal {Q}_{(1,-1,-1)}({\balpha,\boldsymbol{\xi}})
\bigg(\prod_{i,j}\,\tilde F_\delta^\pm (s_{ij})\bigg)\,\mbox{d}(\cdotsp).
}
and this is $B P_{(1,-1,-1),2,4}(\log B)+O\left(B^{1+\eps} \delta^{1-C_{33}\eps}+B^{1-K}\delta^{-C_{34}}\right)$ with $P_{(1,-1,-1),2,4}$ of degree $4$ with leading coefficient
\est{
&
-\frac1{ 2^{3}\cdot 3}\II.
}
\indent
Thus, summarizing for all $\bepsilon\in\{1,-1\}^3$ we have
\est{
M_{\bepsilon,2}^\pm(B)=B P_{\bepsilon,2}(\log B)+O\left(B^{1+\eps} \delta^{1-C_{35}\eps}+B^{1-K}\delta^{-C_{36}}\right)
}
where $P_{\bepsilon,2}$ is a polynomial of degree at most $1$ unless $\bepsilon=(1,-1,-1)$ in which case $P_{\bepsilon,2}$ is of degree $4$ with leading coefficient
\est{
&
-\frac1{ 2^{4}\cdot 3}\II.
}

\subsection{Regrouping the various contributions} By~\eqref{fsp} and~\eqref{emt}, \eqref{emmt} and regrouping the contributions from Sections~\ref{wq},~\ref{wqq} and~\ref{wqqq}, we find 
\est{
N_{\delta}^\pm(B)&=B \sum_{\bepsilon\in\{1,-1\}^3}\big(P_{\bepsilon,1}(\log B)+P_{\bepsilon,2}(\log B)\big)+BP_0+O\left(B^{1+\eps} \delta^{1-C_{37}\eps}+B^{1-K}\delta^{-C_{38}}\right)
\\&=B P_1(\log B))+O\left(B^{1+\eps} \delta^{1-C_{37}\eps}+B^{1-K}\delta^{-C_{38}}\right)
}
where $P_1$ is a polynomial of degree $4$ with leading coefficient
\est{
&\Bigg(\frac1{16}-\frac7{72}-\frac1{ 48}+\sum_{\bepsilon\neq (-1,-1,1),(1,-1,1)}\pr{\frac{\epsilon_1\epsilon_2\epsilon_3}{ 2\cdot (4-\epsilon_1+\epsilon_2)}-\frac{\epsilon_1\epsilon_2\epsilon_3}{ 2\cdot 3\cdot (4-\epsilon_1+2\epsilon_2-\epsilon_3)}}\Bigg)\II=\frac1{48}\II.
}
The estimate~\eqref{casc} and then the Theorem \ref{mtt1} then follows by the final next lemma.

\begin{lemma}
With the notations of Theorem \ref{mtt1} and (\ref{intt}), we have
$
\II=\frac13\mathcal I \mathfrak{S}_1.
$
\end{lemma}
\begin{proof}
First, we observe that for $s^*=1$ we have
\est{
\mathcal {Q}_{\bepsilon}({\balpha,\bzero})&=
\sum_{\bb,\bc,\bff,\bg\in\N^3,\atop h,\ell\in\N}\mu(h) \Bigg(\prod_{k=1}^3\frac{\mu(b_k)\mu(c_k)\mu(f_k)\mu(g_k)}{ r_{1,k}\, r_{2,k}\, r_{3,k}}\Bigg)
\frac{(r_1,\ell)(r_2,\ell)(r_3,\ell)}{\ell^{3}}\varphi(\ell)\cdot 
2\pi^\frac12\prod_{i=1}^3\frac{\Gamma(\frac{1-\alpha_i}2)}{\Gamma(\frac{\alpha_i}2)}\\
&=\Smt\cdot 2\pi^\frac12\prod_{i=1}^3\frac{\Gamma(\frac{1-\alpha_i}2)}{\Gamma(\frac{\alpha_i}2)}
}
where in the second row we computed the Euler product thanks to the lemma~2.7 of~\cite{BBS}.
Therefore, we have
\est{
\II=\frac{\Smt}{4!(2\pi i)^2}\iint_{(c_{\alpha_1},c_{\alpha_2})}\!\!\!\!2\pi^\frac12\bigg(\prod_{i=1}^3\frac{\Gamma(\frac{1-\alpha_i}2)}{\Gamma(\frac{\alpha_i}2)}\bigg)\times
\frac{1}{(2\pi i)^4}\int_{(c_{s_{12}},c_{s_{31}},c_{s_{32}},c_{s_{21}})}\,\frac{\mbox{d}s_{12}\mbox{d}s_{21}\mbox{d}s_{31}\mbox{d}s_{32}}{L(\balpha,s_{12},s_{21},s_{31},s_{32})} \mbox{d}\alpha_1\mbox{d}\alpha_2,
}
where the lines of integrations can be taken at $c_{s_{ij}}=\frac19$, $c_{\alpha_1}=c_{\alpha_2}=\frac23$ and, by~\eqref{invt},
\est{
L(\balpha,s_{12},s_{21},s_{31},s_{32})&=(1-\alpha_1 - s_{21} - s_{31})(1 -\alpha_2  - s_{12} - s_{32})(   -1+\tfrac 32 \alpha_1 -  s_{12} +  s_{21} +  s_{31})\\
&\quad\times(  -1+ \tfrac32 \alpha_2 + s_{12} -  s_{21} +  s_{32})(1 -\tfrac12\alpha_1-\tfrac12 \alpha_2 -  s_{31} -  s_{32 })s_{12}s_{21}s_{31}s_{32}.
}
In the same way as in the end of the proof of Proposition~\ref{mlc}, one has that the inner integral over $s_{12},s_{21},s_{31},s_{32}$ can be evaluated by moving each line of integration to $-\infty$ and collecting the residues of the poles encountered in the process. After this simple but a bit lengthy calculation one finds that the inner integral is equal to $8(\alpha_1\alpha_2\alpha_3(1-\alpha_1)(1-\alpha_2)(1-\alpha_3))^{-1}$, with $\alpha_3=2-\alpha_1-\alpha_2$.
Thus, we finally get
\est{
\II=\frac{\mathfrak{S}_1}{3(2\pi i)^2}\iint_{(c_{\alpha_1},c_{\alpha_2})}2\pi^\frac12
\prod_{i=1}^3\frac{\Gamma(\frac{1-\alpha_i}2)}{\alpha_i(1-\alpha_i)\Gamma(\frac{\alpha_i}2)} \,{\mbox{d}\alpha_2\mbox{d}\alpha_3}=\frac13 \mathfrak{S}_1\cdot\mathcal I
}
by Lemma~\ref{integral}.
\end{proof}
\noindent
\textbf{Acknowledgments.--} The authors are grateful to R\'egis de la Bret\`eche for helpful conversations and comments, for his supervision of the second author during the end of his PhD and for inviting the first author to Paris, where this collaboration started. 

The first author is member of the INdAM group GNAMPA and is partially supported by  PRIN ``Number Theory and Arithmetic Geometry''.

\setcounter{tocdepth}{2}

\bibliographystyle{plainfr}
\bibliography{bibliogr3}

\end{document}